\documentclass[12pt,reqno,oneside]{amsart}

\usepackage{amsmath,amsthm,amssymb,amsfonts,ifpdf,svn}

\ifpdf
\usepackage[pdftex,pdfstartview=FitH,pdfborderstyle={/S/B/W 1},%
colorlinks=true, linkcolor=blue, urlcolor=blue, citecolor=blue,%
pagebackref=true]{hyperref}
\else
\usepackage[hypertex]{hyperref}
\fi

\usepackage[alphabetic,abbrev,msc-links,nobysame]{amsrefs}

\newtheorem {Theorem}   {Theorem}[section]
\newtheorem {thm}    [Theorem]{Theorem}

\newtheorem {lem}      [Theorem]    {Lemma}

\newtheorem {cor}  [Theorem] {Corollary}

\newtheorem {prp}[Theorem]  {Proposition}

\newcounter{AbcT}

\newtheorem {AbcTheorem} [AbcT]{Theorem}

\renewcommand{\a}{\alpha}

\renewcommand{\d}{\delta}

\newcommand{\e}{\varepsilon}
\newcommand{\f}{\varphi}
\newcommand{\g}{\gamma}

\renewcommand{\k}{\kappa}
\renewcommand{\l}{\lambda}
\renewcommand{\o}{\omega}

\newcommand{\s}{\sigma}
\renewcommand{\t}{\theta}

\newcommand{\R}{{\mathbb R}}

\newcommand{\N}{{\mathbb N}}
\newcommand{\Z}{{\mathbb Z}}
\newcommand{\C}{{\mathbb C}}
\newcommand{\E}{{\mathbb E}}
\newcommand{\F}{{\mathbb F}}
\renewcommand{\P}{{\mathbb P}}

\newcommand {\cH} {{\mathcal H}}
\newcommand {\cL} {{\mathcal L}}

\newcommand {\cR} {{\mathcal R}}

\newcommand {\cM} {{\mathcal M}}

\newcommand{\ft}{{\mathfrak t}}

\newcommand {\ua} {{\underline a}}
\newcommand {\ub} {{\underline b}}
\newcommand {\ula} {{\underline \l}}
\newcommand {\up} {{\underline p}}
\newcommand {\ueta} {{\underline \eta}}

\newcommand{\wh}{\widehat}

\newcommand{\be}{\begin{equation}}
\newcommand{\ee}{\end{equation}}
\newcommand{\bea}{\begin{align}}
\newcommand{\eea}{\end{align}}
\newcommand{\bean}{\begin{align*}}
\newcommand{\eean}{\end{align*}}
\newcommand {\IGNORE}[1]  {}

\newcommand {\absolute}[1] {\left| {#1} \right|}

\newcommand {\norm}[1] {\left\| {#1} \right\|}

\newcommand{\fourIdx}[5]{%
\setbox1=\hbox{\ensuremath{^{#1}}}%
\setbox2=\hbox{\ensuremath{_{#2}}}%
\setbox5=\hbox{\ensuremath{#5}}%
\hspace{\ifnum\wd1>\wd2\wd1\else\wd2\fi}%
\ensuremath{\copy5^{\hspace{-\wd1}\hspace{-\wd5}#1\hspace{\wd5}#3}%
_{\hspace{-\wd2}\hspace{-\wd5}#2\hspace{\wd5}#4}%
}}

\newcommand{\Fou}[2]{#1^{\wedge}(#2)}

\newcommand{\HS}{{\rm HS}}
\newcommand{\Id}{{\rm Id}}
\newcommand{\Lip}{{\rm Lip}}

\newcommand{\igap}{\:}
\newcommand{\restr}[1]{|_{#1}}

\DeclareMathOperator{\supp}{supp}

\DeclareMathOperator{\SL}{SL}

\DeclareMathOperator{\SO}{SO}
\DeclareMathOperator{\SU}{SU}

\DeclareMathOperator{\Res}{Res}

\DeclareMathOperator{\Isom}{Isom}
\DeclareMathOperator{\dist}{dist}
\renewcommand{\Re}{\operatorname{Re}}
\DeclareMathOperator{\Tr}{Tr}

\title[Random walks and self-similar measures]%
{Random walks in the group of Euclidean isometries and self-similar measures}
\author{Elon Lindenstrauss}
\address{EL, Einstein Institute of Mathematics,
The Hebrew University of Jerusalem,
Jerusalem, 91904, Israel}
\email{elon@math.huji.ac.il}

\author{P\'eter P. Varj\'u}
\address{PPV, Centre for Mathematical Sciences,
Wilberforce Road, Cambridge CB3 0WA,
England}
\email{pv270@dpmms.cam.ac.uk}

\thanks{EL was supported in part by the ERC (AdG Grant 267259), the ISF (983/09) and the NSF (DMS-0800345). PV was supported in part by the Simons Foundation and the ERC (AdG Grant 267259).}


\SVN $Date$
\date{\SVNDate}

\begin{document}

\begin{abstract}
We study products of random isometries acting on Euclidean space. Building on previous work of the second author,
we prove a local limit theorem for balls of shrinking radius with exponential speed
under the assumption that a Markov operator associated to the rotation component
of the isometries has spectral gap.
We also prove that certain self-similar measures are absolutely continuous with smooth densities.
These families of self-similar measures give higher dimensional analogues of 
Bernoulli convolutions on which absolute continuity can be established for contraction ratios in an open set.
\end{abstract}

\maketitle
%%%%%%%%%%%%%%%%%%%%%%%%%%%%%%%%%%%%%%%%%%
\section{Introduction}
\label{section:introduction}
%%%%%%%%%%%%%%%%%%%%%%%%%%%%%%%%%%%%%%%%%%

%%%%%%%%%%%%%%%%%%%%%%%%%%%%%%%%%%%%%%%%%%%
\subsection{Random walks in Euclidean space}
\label{section:RW}
%%%%%%%%%%%%%%%%%%%%%%%%%%%%%%%%%%%%%%%%%%

We consider two problems in this paper.
The first one concerns random walks in Euclidean space, where the steps are isometries.
Let $X_1,X_2,\ldots\in\Isom(\R^d)$ be a sequence of i.i.d. random orientation preserving
isometries with an arbitrary probability law.
Fix a point $x_0\in \R^d$ and consider the sequence of random points $Y_l=X_l\cdots X_1(x_0)$.
This is called the random walk.
Our aim is to understand the distribution of the point $Y_l$.

The hypothesis in our result will be formulated in terms of spectral properties of an
operator associated to the law of $X_i$.
We denote the canonical projection from the group of isometries to the group of rotations
by $\t:\Isom(\R^d)\to\SO(d)$.
For a function $\f\in L^2(\SO(d))$, we write
\be\label{equation:defT}
T\f(\s)=\E[\f(\t(X_1)^{-1}\s)].
\ee
This defines an operator on the space $ L^2(\SO(d))$.
Denote by  $L^2_0(\SO(d))$ the 1 codimensional subspace of functions orthogonal to the constants.
We say that $T$ has \emph{spectral gap} if there is an integer $l>0$ such that $\|T^l\|_{L^2_0(\SO(d))}<1$.

Our main result on random walks stated below will show that $Y_l$ can be approximated
by Gaussian random variables on scales between $e^{-cl}$ and $\sqrt{l}$.

\begin{thm}\label{theorem:RW}
With notation as above, suppose that $T$ has spectral gap, $d\ge 3$,
$\E[|Y_1|^\a]<\infty$ for some $2<\a\le 4$ and there is no point $x\in\R^d$ such that $X_1(x)=x$
almost surely.
Then there is a point $y_0\in\R^d$, a centrally symmetric Gaussian random variable $Z$
and a number $c>0$ all depending only on the law of $X_1$ such that the following holds.
Let $f$ be a compactly supported smooth function on $\R^d$.
Then
\begin{align*}
\E[f(Y_l)]&=\E[f(\sqrt l Z +y_0)]+O(l^{-\frac{d+\a-2}{2}}+|x_0|^2l^{-\frac{d+2}{2}})\|f\|_1\\
&\qquad{}+O(e^{-cl})\|f\|_{W^{2,(d+1)/2}}.
\end{align*}
The implied constants depend only on the law of $X_1$.
\end{thm}

By centrally symmetric Gaussian variable, we mean one with $0$ mean and covariance matrix $\s\cdot\Id$.
The $L^2$-Sobolev norm in the second error term is defined by
\[
\|f\|_{W^{2,(d+1)/2}}^2=\int|\wh f(\xi)|^2(1+|\xi|)^{d+1}\igap d\xi.
\]

To illustrate the quality of our estimate, we formulate the following
immediate corollary, which we will prove in Section \ref{section:rwproof}.
Denote by $B(r,z)$ the ball of radius $r$ around the point $z\in \R^d$.
\begin{cor}\label{corollary:RW}
Under the assumptions of Theorem \ref{theorem:RW},
there is a point $y_0\in\R^d$
and numbers $A,\s,c>0$ all depending only on the law of $X_1$ such that
\begin{align*}
\P(Y_l\in B(r,z))&=Ar^dl^{-d/2}e^{-|z-y_0|^2/2l\s^2}+O(r^{d+2}l^{-\frac{d+2}{2}})\\
&\qquad{}+O(r^d(l^{-\frac{d+\a-2}{2}}+l^{-\frac{d+2}{2}}|x_0|^2))+ O(e^{-cl}r^{-1/2}).
\end{align*}
\end{cor}

Observe that the error terms are of lower order of magnitude than the main term as long as
$r+|y_0-z|\lesssim\sqrt l$ and  $r>Ce^{-cl}$ .
When $X_1$ is finitely supported, this is optimal up to the constants, since the number of points $Y_l$ can
attain grows exponentially.

A theorem similar to Theorem \ref{theorem:RW} has been given in \cite{Var-Euclidean}*{Theorem 3}.
That result holds in greater generality but provides weaker error terms.
In particular, when $d\ge 3$, and we replace the spectral gap condition by requiring
merely that the support of $\t(X_1)$ generates a dense subgroup of $\SO(d)$, one can get the
same conclusion with the second error term replaced by $O(e^{-c{l^{1/4}}})$.

There are no examples known to the authors when  $\supp\t(X_1)$ generates a dense subgroup of $\SO(d)$,
and the operator $T$ does not have spectral gap.
In fact, it is possible that the above condition about denseness implies the spectral gap condition, but this
is not known in general.
However,  Bourgain and Gamburd proved this in the following important special case.

\begin{AbcTheorem}[\cite{BG-SU2}]\label{theorem:SO3}
Let $X$ be a random element of $\SO(3)$ and suppose that $\supp X$ is finite and consists of matrices with
algebraic entries.
Suppose further that $\supp X$ generates a dense subgroup.
Then the operator
\[
T f(g)=\E(f(Xg))
\]
acting on $L^2(\SO(3))$ has spectral gap.
\end{AbcTheorem}

This theorem has been generalized to $\SU(d)$, $d\ge2$ in a subsequent paper \cite{BG-SUd}, and has been extended
 very recently by Benoist and Saxc\'e to general simple compact Lie groups and in particular for $\SO(d)$, $d\ge3$ \cites{Benoist-deSaxce,deSaxce-product}.
The condition on algebraicty seems essential for the argument and its removal would probably
require significant new ideas.

The problem of studying random walks on $\Isom(\R^d)$ can be traced back to a paper of Arnold and Krylov
\cite{AK-uniform}.
A central limit theorem (describing the behaviour of $Y_l$ on scale $\sqrt l$) was given by Tutubalin \cite{Tut-CLT}
in the $d=3$ case and has been subsequently generalized by several authors.
A ratio limit theorem (describing the behaviour of $Y_l$ on scale $1$) was given by Kazhdan \cite{Kaz-uniform}
and Guivarc'h \cite{Gui-uniform} in the $d=2$ case.
For further details about the history of the problem we refer to \cite{Var-Euclidean} and its references.

%%%%%%%%%%%%%%%%%%%%%%%%%%%%%%%%%%%%%%%%%%
\subsection{Self-similar measures}
\label{section:selfsimilar}
%%%%%%%%%%%%%%%%%%%%%%%%%%%%%%%%%%%%%%%%%%

 The second problem studied in this paper is the smoothness of self-similar measures.
Let $\eta$ be a probability measure supported on contractive similarities of $\R^d$.
A contractive similarity is a map of the form $x\mapsto \l\cdot\s(x)+v$, where $0<\l<1$, $\s\in\SO(d)$ and $v\in\R^d$.
Let $\nu$ be a probability measure on $\R^d$ and let $X$ be a random similarity with law $\eta$
and $Y$ be an independent random point $Y\in\R^d$ with law $\nu$.
Suppose that $\E(|X(0)|)<\infty$.
We say that $\nu$ is \emph{$\eta$-stationary}, if the law of $X(Y)$ is also $\nu$.
It is easily seen that for every $\eta$ supported on contractive similarities, there is a unique
$\eta$-stationary measure.
A measure $\nu$ is called self-similar if it is $\eta$-stationary for some probability measure $\eta$ supported
on contractive similarities.
For general properties of self-similar measures we refer to \cite{Hut-self-similar}.

An extensively studied class of self-similar measures are the  Bernoulli convolutions
introduced by Jessen and Wintner \cite{JW-Bernoulli} in the 30's.
Let $0<\l<1$ be a number and let $\nu_\l$ be the law of the random power series
$\sum_{n=0}^{\infty}A_n\l^n$,
where $A_n$ are independent Bernoulli random variables such that $\P(A_n=1)=\P(A_n=-1)=1/2$ for all $n$.
It is easily seen that $\nu_l$ is self-similar:
Take $\eta_\l$ to be the probability measure supported on the two similarities $x\mapsto \l x\pm 1$ putting $1/2$
mass on each.
Then $\nu_\l$ is $\eta_\l$-stationary.

It is easily seen that $\nu_\l$ is a singular measure supported on a Cantor set if $\l<1/2$.
Moreover, $\nu_{1/2}$ is the normalized Lebesgue measure restricted to the interval $[-2,2]$.
This follows from the fact that almost all numbers in that interval have a unique binary expansion.
When $\l>1/2$, $\nu_\l$ is more mysterious.
Erd\H os \cite{Erd-Bernoulli1}, \cite{Erd-Bernoulli2} studied the regularity properties of $\nu_\l$.
He showed that there is a number $\l_0<1$ such that for almost all $\l>\l_0$, $\nu_\l$ is absolutely continuous.
This was extended to almost all $\l>1/2$ by Solomyak \cite{Sol-Bernoulli}.
A remarkable recent advance was made by Hochman and Shmerkin \cites{Hoc-Bernoulli, Shm-Bernoulli}; specifically 
Shmerkin shows (based on the result of Hochman) that the set of parameters $1/2<\l<1$,
for which  $\nu_\l$ is not absolutely continuous with respect to Lebesgue measure is of
Hausdorff dimension 0.
On the other hand, Erd\H os observed that $\nu_\l$ is singular if $\l^{-1}$ is a Pisot number,
e.g. $\l=(\sqrt5-1)/2$.
It is a long standing open problem whether there is a number $\bar\l<1$
such that $\nu_\l$ is absolutely continuous  for {\em all} $\l>\bar\l$.
Moreover, it is not known whether in the interesting range $1/2<\l<1$ there are any examples for singular $\nu_\l$ apart from
those when $\l^{-1}$ is Pisot.
For more on Bernoulli convolution we refer to \cite{PSS-survey}.

Unfortunately, Bernoulli convolutions are not amenable to our methods.
However, we can answer some analogues of this question in dimension 3 and above.
If $\k: x\mapsto \l\cdot\s(x)+v$, then we write $\l(\k)=\l$ and $\t(\k)=\s$, and $g(\k)$
denote the isometry $x\mapsto\s(x)+v$.
We note that $g(\k)$ is not a homomorphism and depend on our choice of origin.

\begin{thm}\label{theorem:selfsimilar}
Let $\eta$ be a probability measure supported on finitely many contracting similarities of $\R^d$
without a common fixed point for some $d\ge 3$.
Let $X$ be a random similarity with law $\eta$ and suppose that 
the operator
\[
Tf(\s)=\E(f(\t(X)^{-1}\s))
\]
on $L^2(\SO(d))$ has spectral gap.
Let $n$ be an integer.
Then there is a number $\bar\l<1$ such that the unique $\eta$-stationary measure is absolutely continuous
with $n$ times differentiable density, if $\l(X)>\bar\l$ almost surely.
The number $\bar\l$ depends on $d$, $n$, the spectral gap for $T$, the cardinality of $\supp\eta$ and the minimal value of the probabilities $\P(X=g)$ for $g \in \supp\eta$.
\end{thm}

There is an analogy between self-similar measures and Furstenberg measures associated to random walks on
non-compact semisimple Lie groups.
Bourgain proved results about the absolute continuity of  Furstenberg measures, which are related to
Theorem \ref{theorem:selfsimilar}.
We refer the reader to the papers \cite{Bourgain-SL2-ac, Bourgain-expansion-AB} for more details.

%%%%%%%%%%%%%%%%%%%%%%%%%%%%%%%%%%%%%%%%%%
\subsection{Some ideas of the proofs}
\label{section:idea}
%%%%%%%%%%%%%%%%%%%%%%%%%%%%%%%%%%%%%%%%%%

We outline the main ideas in the paper.
We define a family of operators $S_r$ for $r>0$ acting on the space $L^2(S^{d-1})$ that provide
a natural link between the two problems explained above.
These operators were introduced by Kazhdan \cite{Kaz-uniform} and Guivarc'h \cite{Gui-uniform} in their
works of studying  random walks on $\Isom(\R^2)$.

Let $X\in\Isom(\R^d)$ be a random element and write $v(X)\in\R^d$ for its translation part
and $\t(X)\in\SO(d)$ for its rotation part.
Let $\f\in L^2(S^{d-1})$.
Then we write
\[
S_r\f(\xi)=e^{-2\pi ir\langle\xi,v(X)\rangle}\f(\t^{-1}(X)\xi).
\]
This defines the operator $S_r$, which depends on the law of $X$.
We defer the more detailed discussion of these operators and their relation to random walks and self-similar measures  to Section~ \ref{section:notation}.
In the context of self-similar measures,
we use the ``projection'' $g_*\eta$ of $\eta$ to the isometry group in place of the law of $X$.

In this paper we will prove (see Theorem \ref{theorem:technical} in Section \ref{section:notation})
the norm estimates
\be\label{equation:normestimate}
\|S_r\|\le1-c\min\{1,r^2\}
\ee
with some constant $c$ depending only on the law of $X$ in the notation of Section \ref{section:RW}.
Both Theorems \ref{theorem:RW} and \ref{theorem:selfsimilar} can be deduced from
\eqref{equation:normestimate}.

Guivarc'h \cite{Gui-uniform} proved in the $d=2$ case the estimates
\be\label{equation:ineffective}
\|S_r\|\le1-cr^2\quad \text{for $r<1$ and}\quad
\|S_r\|\le1-c_r\quad  \text{for $r\ge1$}
\ee
with a number $c_r$ depending also on $r$.
However, this argument depends on the special feature of the two dimensional case that $\SO(2)$
is commutative.

In a more recent paper Conze and Guivarc'h \cite{CG-spectralgap}*{Theorem 4.6}
proved the estimates \eqref{equation:ineffective} in higher dimension under the assumption
that the operator $T$ as defined in \eqref{equation:defT} has spectral gap.
This is the same assumption as ours in Theorem \ref{theorem:technical}, however, we obtain
the uniform estimate \eqref{equation:normestimate}, which is needed for both of our applications.

The paper \cite{Var-Euclidean} also studies the operators $S_r$.
In that paper no spectral gap assumption is posed, instead, it is based on a weaker property
that can be verified in great generality.
On the other hand, the estimates obtained there are much weaker than \eqref{equation:normestimate}.
In fact that paper only provides bounds for $\|S_r\f\|_2$ which depend also on the Lipschitz norm
of $\f$.
To avoid technicalities we do not give the details here, just refer to the paper \cite{Var-Euclidean}.
We also note that the method in \cite{Var-Euclidean}
does not seem to be strong enough to give \eqref{equation:normestimate}
even under the spectral gap assumption.
To indicate the improvement achieved by the methods of the current paper, we note that
the methods of \cite{Var-Euclidean}
can prove Theorem \ref{theorem:RW} only with the second error term replaced
by $O(e^{-cl^{1/2}})$.

The operators $S_r$ are related to the operator $T$ defined in \eqref{equation:defT} and
they are amenable to the method of Bourgain and Gamburd \cite{BG-SU2}, \cite{BG-SUd}.
We adapt this method to the problem we consider.
This will be discussed in more detail later.
Now we mention only the most crucial new ingredient in our proof, which is an estimate of the following type:
\be\label{equation:introNC}
\P(Y_l\in B(r,y))\le C r^{c_d}.
\ee
Here $Y_l$ is the random walk as defined in Section \ref{section:RW}, $r>0$ is a number, $B(r,y)$
is the ball of radius $r$ around a point $y\in\R^d$, $l>C\log (r^{-1})$ is an integer,
$C$ is a number depending on the law of $X_1$ and $c_d>0$ is a number depending on the dimension.
In words, this means that after $\log(r^{-1})$ steps, the probability that the random walk is
in a given ball of radius $r$ is bounded by a polynomial of $r$.

If we assumed that the support of $X_1$ is concentrated on isometries which have both rotation and
translation parts algebraic, then \eqref{equation:introNC} would follow from simple Diophantine considerations.
In fact, this is very related to how algebraicity is used by Bourgain and Gamburd in their proof of Theorem \ref{theorem:SO3}.
Without Diophantine assumptions \eqref{equation:introNC} is more difficult, and requires new ideas.

To establish \eqref{equation:introNC}, we estimate the Fourier transform of the law of $Y_l$,
that is the function $\E(e^{-2\pi i\langle\xi,Y_l\rangle})$.
The required estimate on the Fourier transform would follow immediately from the norm estimates
\eqref{equation:normestimate};
but our argument works in the opposite direction, and uses \eqref{equation:introNC} to prove \eqref{equation:normestimate} and so this does not help us.
What does help us is the following weaker statement for which there is a relatively simple direct proof:
If $R_1>R_2>0$ are two numbers such that $|R_1-R_2|<c$, where $c$ is a number depending only
on the law of $X_1$, then $\|S_{R_i}\|<1-c|R_1-R_2|^2$ for $i=1$ or $i=2$.
That is, we are able to establish \eqref{equation:normestimate} for at least one of two nearby values
of the parameter.

This statement allows us to estimate $\E(e^{-2\pi i\langle\xi,Y_l\rangle})$ on one of  two nearby spheres.
Then we use a simple fact which holds for all probability measures on $\R^d$: if the $L^2$-average of
the Fourier transform is small on a sphere then it is also small on nearby spheres.
This can be verified by decomposing the measure into two parts, one which has  Fourier transform
of small Lipschitz norm and one whose average Fourier transform on spheres decays fast.
This will conclude the proof of \eqref{equation:introNC}.

Finally, we mention that in the paper \cite{LV-affine} we prove an analogue of the results of this paper
in the group $\SL_d(\F_p)\ltimes \F_p$.
That paper follows a similar scheme and exhibits some of the ideas of this paper in a technically easier
setting.

%%%%%%%%%%%%%%%%%%%%%%%%%%%%%%%%%%%%%%%%
\subsection{Organization of the paper}
%%%%%%%%%%%%%%%%%%%%%%%%%%%%%%%%%%%%%%%%
In Section \ref{section:notation}, we introduce some notation that will be used throughout the
paper, in particular, we explain the operators $S_r$ in more details.
We will also state there a technical result, Theorem \ref{theorem:technical}, which will be used later
to deduce both Theorems \ref{theorem:RW} and \ref{theorem:selfsimilar}.
Sections \ref{section:preliminary}--\ref{section:bourgaingamburd2} are devoted to the proof of Theorem \ref{theorem:technical}.
Section \ref{section:preliminary} provides some preliminary norm estimates that we mentioned above, that is,
we bound the norm of $S_r$ for one of two nearby values of the parameter.
In Section \ref{section:decay} we prove the non-concentration estimate \eqref{equation:introNC}.
We provide some background material on sets and measures of ``large dimension" in Section \ref{section:compactgroups}.
This is not very new, but it is unavailable in the literature in the form we need it.
In Sections \ref{section:bourgaingamburd1} and \ref{section:bourgaingamburd2}
we recall the Bourgain-Gamburd method and finish the
proof of Theorem \ref{theorem:technical}.
As we will see in the next section, Theorem \ref{theorem:technical} is stated under some convenient
simplifying assumptions.
We reduce the general situation to this special setting in Section \ref{section:general},
where we formulate and prove Corollary \ref{corollary:general}.
Finally in Sections \ref{section:rwproof} and \ref{section:ssproof} we deduce Theorems \ref{theorem:RW}
and \ref{theorem:selfsimilar} respectively.
Deducing Theorem \ref{theorem:selfsimilar} from our spectral gap estimate \eqref{equation:normestimate} 
is much simpler in the case when all the contraction
factors equal.
We first give the proof of this case, and we also give an estimate on $\bar\l$, see Theorem \ref{theorem:onelambda}.
Then we turn to the general case, where we also use a result of Ab\'ert \cite{Abe-spectral-gap}.

%%%%%%%%%%%%%%%%%%%%%%%%%%%%%%%%%%%%%%%%%%
\subsection*{Acknowledgement}
%%%%%%%%%%%%%%%%%%%%%%%%%%%%%%%%%%%%%%%%%%
We thank Mikl\'os Ab\'ert for helpful discussions and for making his paper \cite{Abe-spectral-gap} available
to us before its publication.
His result enabled us to treat self-similar measures with varying contraction ratios.
We are also grateful to the referees for their careful reading of our paper.

Part of this work was conducted while E.L. was a fellow at the Israeli Institute for Advanced Studies. E.L. would like to thank the Institute for providing ideal working conditions.

%%%%%%%%%%%%%%%%%%%%%%%%%%%%%%%%%%%%%%%%%%
\section{Notation}
\label{section:notation}
%%%%%%%%%%%%%%%%%%%%%%%%%%%%%%%%%%%%%%%%%%

We identify the group of orientation preserving isometries of the $d$-dimensional
Euclidean space with the
semidirect product $\Isom(\R^d)=\R^d\rtimes \SO(d)$.
For $g=(v,\t)\in\R^d\rtimes \SO(d)$ and a point $x\in\R^d$
we write
\[
g(x)=v+\t x,
\]
and we define the product of two isometries by
\[
(v_1,\t_1)(v_2,\t_2)=(v_1+\t_1 v_2,\t_1\t_2).
\]
If $g$ is an isometry, we write $v(g)$ for the translation
component and $\t(g)$ for the rotation component of $g$ in the
above semidirect decomposition.
With this notation, the inverse of $g$ is given by the formula
\[
g^{-1}(x)=-\t(g)^{-1}v(g)+\t(g)^{-1}x.
\]

Let $\mu\in\cM(\Isom(\R^d))$, that is a probability measure on $\Isom(\R^d)$.
Define the convolution $\mu*\mu$ in the usual way by
\[
\int_{\Isom(\R^d)} f(g)\igap d\mu*\mu(g)
=\int_{\Isom(\R^d)}\int_{\Isom(\R^d)} f(g_1g_2)\igap d\mu(g_1)d\mu(g_2),
\]
for $f\in C(\Isom(\R^d))$ and
write
\[
\mu^{*(l)}=\underbrace{\mu*\cdots *\mu}_{l-{\rm fold}}
\]
for the $l$-fold convolution.
With this notation, $\mu^{*(l)}$ is the distribution of the
product of $l$ independent random element of $\Isom(\R^d)$
with law $\mu$.
We define the measure $\check\mu$ by the formula
\begin{equation}\label{eq_defcheck}
\int_{\Isom(\R^d)}f(g)\igap d\check\mu(g)
=\int_{\Isom(\R^d)}f(g^{-1})\igap d\mu(g),
\end{equation}
for $f\in C(\Isom(\R^d))$ and say that $\mu$ is \emph{symmetric} if $\check\mu=\mu$.
The measure $\mu$ also acts on measures on $\R^d$ in the following way:
If $\nu\in\cM(\R^d)$, we can define another measure
$\mu.\nu$ on $\R^d$ by:
\be\label{equation:mudotnu}
\int_{\R^d} f(x)\igap d\mu.\nu(x)
=\int_{\Isom(\R^d)}\int_{\R^d}f(g(x))\igap d\mu(g)d\nu(x),
\ee
for $f\in C(\R^d)$.

We write $\d_{x_0}$ for the Dirac delta measure concentrated
at the point $x_0$.
With this notation, the law of the $l$th step of the random walk is
\[
\mu^{*(l)}.\d_{x_0}.
\]
It a simple calculation to check that
\[
\mu^{*(l+1)}.\d_{x_0}=\mu.(\mu^{*(l)}.\d_{x_0}).
\]
Hence our main goal is to understand the operation $\nu\mapsto\mu.\nu$.

This is achieved by studying the Fourier transform, which is given
by the formula
\[
\wh{\nu}(\xi)=\int e(\langle\xi, x\rangle)\igap d\nu(x),
\]
where $e(x):=e^{-2\pi ix}$.
For the Fourier transform of $\mu.\nu$ we get
\begin{align}
\Fou{(\mu.\nu)}{\xi}
&=\int e(\langle\xi,g(x)\rangle) \igap d\mu(g)d\nu(x)\nonumber\\
&=\int e(\langle\xi,v(g)+\t(g)(x)\rangle)
\igap d\mu(g)d\nu(x)\nonumber\\
\label{equation:Fourier}
&= \int e(\langle\xi,v(g)\rangle)
\wh{\nu}(\t(g)^{-1}\xi)\igap d\mu(g).
\end{align}

This formula shows that the action of $\mu$ on the Fourier transform
of $\nu$ can be disintegrated with respect to spheres centered at
the origin.
For every $r\ge0$,
we define a unitary representation of the group $\Isom(\R^d)$
on the space $L^2(S^{d-1})$.
Let
\be\label{eq_defrho}
\rho_r(g)\f(\xi)=e(r\langle\xi,v(g)\rangle)
\f(\t(g)^{-1}\xi)
\ee
for $g\in\Isom(\R^d)$, $\f\in L^2(S^{d-1})$ and $\xi\in S^{d-1}$.
We denote the character appearing in the definition of $\rho_r(g)$ by
$\omega_r(g): S^{d-1}\to \C$,
\[
\omega_r(g)(\xi):=e(r\langle\xi,v(g)\rangle).
\]
We also define the operator
\be\label{eq_defSr}
S_r(\f)=\int\rho_r(g)(\f)\igap d\mu(g).
\ee

For a function $\f\in C(\R^d)$ and $r\ge0$,
we denote by $\Res_r\f$ its restriction
to the sphere of radius $r$.
I.e. $\Res_r: C(\R^d)\to C(S^{d-1})$ is an operator
defined by
$[\Res_r \f](\xi)=\f(r\xi)$
for $|\xi|=1$.
With this notation, we can write \eqref{equation:Fourier} as
\be\label{equation:Sr}
\Res_r(\wh{\mu.\nu})(\xi)=S_r(\Res_r \wh\nu)(\xi).
\ee

We denote by $\cR$
the (left) regular representation of $\SO(d)$, that is we write
\[
\cR(\t)\f(\s)=\f(\t^{-1}\s)
\]
for a function $\f\in L^2(\SO(d))$.
In addition, we denote by $\cR_0$ the restriction of $\cR$ to $L_0^2(\SO(d))$.

We recall from Section \ref{section:RW} that $T$ is an operator acting on the
space $L^2(\SO(d))$ defined by
\[
T\f(\s)=\int \f(\t(g)^{-1}\s)\igap d\mu(g).
\]
More generally, if $\pi$ is a representation of a group $G$ and $\mu$ is a probability measure
on it, then we write
\be\label{equation:piofmu}
\pi(\mu)=\int\pi(g) \igap d\mu(g),
\ee
which is an operator on the representation space of $\pi$.
With this notation, we have $S_r=\rho_r(\mu)$ and $T=\cR(\t(\mu))$.

Now we formulate the main technical result of the paper.
We will use this to deduce the theorems stated in the introduction.
We make some simplifying assumptions, which make our statements and calculations
easier.
These are not serious restrictions of generality, and the general case can be
reduced to the special case when these conditions hold.
The following assumptions are assumed to hold throughout Sections~\ref{section:preliminary}--\ref{section:bourgaingamburd2}:
\be
\mu=\check\mu_0*\mu_0\label{equation:symmetric}
\ee
for some $\mu_0\in\cM(\Isom(\R^d))$,
\begin{align}
\|T\f\|_2&\le\frac{\|\f\|_2}{2},\qquad\text{for all $\f\in L^2_0(\SO(d))$}\label{equation:gap}\\
\int v(g)d\mu(g)&=0,\label{equation:centered}\\
\int |v(g)|^2d\mu(g)&=1,\label{equation:moment2}\\
M:=\int |v(g)|^3d\mu(g)&<\infty.\label{equation:moment3}
\end{align}

We note that \eqref{equation:centered} and \eqref{equation:moment2} amounts to
a suitable choice of origin and normalization.
The conditions \eqref{equation:symmetric} and \eqref{equation:gap} can be satisfied, if $\mu$
is such that $T$ has spectral gap, and we replace it by $\check\mu^{*(l_0)}*\mu^{*(l_0)}$
for a suitably large integer $l_0$.
Finally, \eqref{equation:moment3} can be satisfied by restricting $\mu$ to any large ball in $\Isom(\R^d)$.
Details of these ideas will be given in Section \ref{section:general}.

\begin{thm}\label{theorem:technical}
Let $\mu\in\cM(\Isom(\R^d))$ and
suppose that the assumptions \eqref{equation:symmetric}--\eqref{equation:moment3} hold.

Then there is a number $c$, which depends only on $d$
such that
\[
\|S_r\|\le 1-c\min\{r^2,M^{-2}\}.
\]
\end{thm}

We comment on the role of the third moment $M$.
Consider the following case.
Suppose that
\[
\mu(g:v(g)=0)=1-M^2\quad{\rm and}\quad\mu(g:|v(g)|=M)=M^{-2}.
\]
Then both conditions \eqref{equation:moment2} and \eqref{equation:moment3} hold.
If $\f\equiv1$ is a constant function, then it is easily seen that $\langle S_r\f,\f\rangle\ge 1-2M^2$ for all $r$.
This example shows that it is not possible to bound $\|S_r\|$ using only the conditions
 \eqref{equation:symmetric}--\eqref{equation:moment2}, but one needs to control the
probability that a random isometry with law $\mu$ fixes or approximately fixes a given point.
The role of $M$ is to control this degeneracy in a quantitative way.
The example also show that the dependence on $M$ is optimal up to the constant.

Throughout the paper the letters $c,C$ and various subscripted versions
refer to constants and parameters.
They are allowed to depend on $d$ but not on other parameters, unless the contrary is stated.
The same symbol occurring in different places need not have the same
value unless the contrary is explicitly stated.
For convenience, we use lower case for constants which are best
thought of to
be small and upper case for those which are best
thought of to be large.

%%%%%%%%%%%%%%%%%%%%%%%%%%%%%%%%%%%%%%%%%%
\section{Preliminary spectral gap estimates}
\label{section:preliminary}
%%%%%%%%%%%%%%%%%%%%%%%%%%%%%%%%%%%%%%%%%%

In this section, we prove that for any number $r>0$, there is at most one
eigenvalue of $S_r$ close to 1, and we estimate the norm of $S_r$ on the orthogonal complement of the
corresponding eigenfunction.
Moreover, if such an eigenvalue exists for some $r=r_1$ then it can not exist for $r=r_2$ if
$r_2$ is not too close to and not too far from $r_1$.
We make this precise below in a proposition.
This result is related to the paper \cite{CG-spectralgap}.

\begin{prp}\label{proposition:preliminary}
Suppose that the assumptions \eqref{equation:symmetric}--\eqref{equation:moment3} hold.
Let $r\ge 0$, and let $\f_1,\f_2\in L^2(S^{d-1})$  be orthonormal functions.
Then
\[
\|S_r\f_i\|_2\le1-c,
\]
for $i=1$ or $i=2$, where $c>0$ is an absolute constant.

In addition, there is a number $c>0$ depending only on the dimension $d$ such that the following holds:
Let $r_1,r_2\ge 0$ such that $|r_1-r_2|\le c M^{-1}$.
Then
\[
\|S_{r_i}\|\le1-c|r_1-r_2|^2
\]
holds for $i=1$ or $i=2$.
\end{prp}

We say that a function $\f\in L^2(S^{d-1})$ of unit norm is \emph{$\e$-invariant}
for~$S_r$ if $\|S_r\f\|_2\ge 1-\e^2$.
The following lemma explains the terminology.
\begin{lem}\label{lemma:einvariance}
Let $\f\in L^2(S^{d-1})$ be a function of unit norm that is $\e$-invariant for $S_r$.
Then
\[
\langle \f-S_r\f,\f\rangle\le2\e^2\quad{\rm and}\quad\|\f-S_r\f\|_2\le2\e.
\]
\end{lem}

Before giving the proof of the lemma, we explain the organization of the rest of the section.
The proof of Proposition \ref{proposition:preliminary} is based on the observation that if $\psi_1$ and $\psi_2$
are $\e$-invariant for $S_{r_1}$ and $S_{r_2}$, respectively,
then $\psi_1\overline{\psi_2}$ is $C\e$-invariant for $S_{r_1-r_2}$.
This will be proved in Section \ref{section:approximating}.
Since the product of two $L^2$ functions are not in $L^2$ in general, it will be convenient
to approximate the $\e$-invariant functions with bounded functions.
The relevant estimates are given in Section \ref{section:constantModulus}.
In Section \ref{section:Fr}, we bound the norm of $S_r$ for small values of $r$.
Finally, we put together the proof of the proposition from the above three components in Section
\ref{section:proofPreliminary}.

Since all irreducible representations of $\SO(d)$ are contained in its regular representation,
assumption \eqref{equation:gap} implies the following spectral gap estimate on $S_0$:
\be\label{equation:gapS0}
\|S_0\f\|_2\le \frac{\|\f\|_2}{2},\quad\text{for all $\f\in L^2_0(S^{d-1})$}.
\ee
We only need, in fact, this weaker assumption in this section.

\begin{proof}[Proof of Lemma \ref{lemma:einvariance}]
By assumption \eqref{equation:symmetric} it follows that $S_r$ is a non-negative
selfadjoint operator, hence it has a square root $X$, that is $S_r=X^2$.
Then
\[
\|X\f\|_2\ge\|X^2\f\|_2\ge1-\e^2
\]
and
\[
\langle S_r\f,\f\rangle=\langle X\f,X\f\rangle\ge (1-\e^2)^2\ge1-2\e^2,
\]
which proves the first claim.

For the second claim, we write
\[
\langle \f-S_r\f,\f-S_r\f \rangle
=\langle \f-S_r\f,\f \rangle-\langle \f,S_r\f \rangle+\langle S_r\f,S_r\f \rangle.
\]
Since $S_r$ is non-negative of norm at most 1, 
\[
\langle S_r\f,S_r\f \rangle=\langle \f,S_r^2\f \rangle\le\langle \f,S_r\f \rangle.
\]
Thus
\[
\|\f-S_r\f\|_2^2=\langle \f-S_r\f,\f-S_r\f \rangle
\le\langle \f-S_r\f,\f \rangle\le2\e^2.
\]
\end{proof}

%%%%%%%%%%%%%%%%%%%%%%%%%%%%%%%%%%%%%%%%%%%
\subsection{Reducing to the case of constant modulus}
\label{section:constantModulus}
%%%%%%%%%%%%%%%%%%%%%%%%%%%%%%%%%%%%%%%%%%%%%
We prove in this section that any $\e$-invariant function $\f$ for $S_r$ 
can be approximated by one which has constant modulus.
The key idea is an application of the spectral gap estimate \eqref{equation:gapS0}
to the function $|\f|$.

\begin{lem}\label{lemma:absolute}
Let $\e,r\ge0$ and let $\f\in L^2(S^{d-1})$ be a function of unit norm, which is $\e$-invariant for $S_r$.
Then
\[
\psi:=\frac{\f}{|\f|}
\]
is $C\e$-invariant for $S_r$ and $\|\f-\psi\|_2\le C\e$, 
where $C$ is an absolute constant.
\end{lem}

\begin{proof}
Since $|\omega_r|\equiv1$, we have
\begin{align}
\|S_r\f\|_2
&=\left\|\int\omega_r(g)\cdot\rho_0(g)\f\igap d\mu(g)\right\|_2\nonumber\\
&\le\left\|\int\rho_0(g)|\f|\igap d\mu(g)\right\|_2
=\|S_0|\f|\|_2.\label{equation:absolute}
\end{align}

We can write $|\f|=A+(|\f|-A)$, where
\[
A=\int_{S^{d-1}}|\f(\xi)|\igap d\xi.
\]
Since
\[
\int_{S^{d-1}}|\f(\xi)|-A\igap d\xi=0,
\]
we can use assumption \eqref{equation:gapS0} to get
\[
\|S_0|\f|\|_2^2\le A^2+\frac{1}{4}\||\f|-A\|_2^2=\frac{3A^2+1}{4}.
\]

We compare this with \eqref{equation:absolute} and use that $\f$ is $\e$-invariant.
We get
\[
1-2\e^2\le(1-\e^2)^2\le\frac{3A^2+1}{4}
\]
and hence
\[
A\ge A^2\ge 1-\frac{8}{3}\e^2.
\]
In addition, we have
\[
\||\f|-A\|_2=(1-A^2)^{1/2}\le\frac{\sqrt 8}{\sqrt 3}\e.
\]

Then we write
\[
\|\f-\psi\|_2
=\left\|\f-\frac{\f}{|\f|}\right\|_2
=\||\f|-1\|_2
\le\||\f|-A\|_2+(1-A)\le C\e.
\]
For the last inequality, we used $\e^2\le\e$.

For $\e$-invariance, we will prove
\[
\|S_r\psi\|_2\ge\langle S_r\psi,\f\rangle=\langle \psi,S_r\f\rangle\ge1-C\e^2.
\]
Only the last inequality is non-trivial.
For that, we write
\[
\langle \psi,S_r\f\rangle
=\langle \psi,\f\rangle-\langle \psi-\f,\f-S_r\f\rangle-\langle \f,\f-S_r\f\rangle.
\]
We have
\[
\langle \psi,\f\rangle=\int \frac{\f(\xi)}{|\f(\xi)|}\cdot\overline{\f(\xi)}\igap d\xi=A\ge1-C\e^2.
\]
By the Cauchy-Schwartz inequality and using Lemma \ref{lemma:einvariance}:
\[
|\langle \psi-\f,\f-S_r\f\rangle|\le C\e^2.
\]
Then we use  Lemma \ref{lemma:einvariance} again:
\[
\langle \f,\f-S_r\f\rangle\le C\e^2.
\]
Combining these inequalities, we get
\[
\langle \psi,S_r\f\rangle
=1-C\e^2,
\]
as required.
\end{proof}

%%%%%%%%%%%%%%%%%%%%%%%%%%%%%%%%%%%%%%%%%%%%%
\subsection{An \texorpdfstring{$\e$}{epsilon}-invariant function for \texorpdfstring{$S_{r_1-r_2}$}{S(r1-r2)}}
\label{section:approximating}
%%%%%%%%%%%%%%%%%%%%%%%%%%%%%%%%%%%%%%%%%%%%%%

We fix $r_1\ge r_2\ge 0$ and functions $\psi_1,\psi_2\in L^2(S^{d-1})$
which are $\e$-invariant for $S_{r_1}$ and $S_{r_2}$ respectively.
This is possible only, if $\rho_{r_i}(g)\psi_i$ is very close to $\psi_i$ for most $g$ in the
support of $\mu$.
(See Lemma \ref{lemma:approximating2} below.)
Then $\rho_{r_1-r_2}(g)\psi_1\overline{\psi_2}$ is very close to  $\psi_1\overline{\psi_2}$.
Hence $\psi_1\overline{\psi_2}$ is $\e'$-invariant for $S_{r_1-r_2}$.
We can get $\e'=C\e$ as the following lemma shows.

\begin{lem}\label{lemma:approximating}
Let $r_1\ge r_2\ge 0$ and $\psi_1,\psi_2\in L^2(S^{d-1})$ be two functions
such that $|\psi_1|\equiv|\psi_2|\equiv1$.
Suppose that $\psi_1$ and $\psi_2$ are $\e$-invariant for $S_{r_1}$ and $S_{r_2}$ respectively,
where $\e>0$ is a number.
Then the function $\psi_1\overline{\psi_2}$ is $C\e$-invariant for $S_{r_1-r_2}$.
\end{lem}

We record a useful identity in the next lemma.

\begin{lem}\label{lemma:approximating2}
Let $r\ge0$ and let $\psi\in L^2(S^{d-1})$ be a function of $L^2$-norm~1.
Then
\[
\int\|\rho_r(g)\psi-\psi\|_2^2 \igap d\mu(g)
= 2\langle\psi- S_r\psi,\psi\rangle.
\]
\end{lem}
\begin{proof}
We note the identity
\begin{align*}
\|\rho_r(g)\psi-\psi\|_2^2
&=2-\langle\rho_r(g)\psi,\psi\rangle-\langle\psi,\rho_r(g)\psi\rangle\\
&=2-\langle\rho_r(g)\psi,\psi\rangle-\langle\rho_r(g^{-1})\psi,\psi\rangle.
\end{align*}
By assumption \eqref{equation:symmetric}, $\mu$ is symmetric, hence
\[
2-2\langle S_r\psi,\psi \rangle
=\int(2-2\langle \rho_r(g)\psi,\psi\rangle) \igap d\mu(g)
=\int\|\rho_r(g)\psi-\psi\|_2^2 \igap d\mu(g).
\]
\end{proof}

\begin{proof}[Proof of Lemma \ref{lemma:approximating}]
We can write
\begin{align}
S_{r_1-r_2}(\psi_1\overline{\psi_2})
&=\int\rho_{r_1}(g)\psi_1\cdot\overline{\rho_{r_2}(g)\psi_2}\igap d\mu(g)\nonumber\\
&=\psi_1\overline{\psi_2}\label{equation:mainterm}\\
&\qquad{}+\int(\rho_{r_1}(g)\psi_1-\psi_1)\cdot\overline{\psi_2}\igap d\mu(g)
\label{equation:error1}\\
&\qquad{}+\int\psi_1\cdot\overline{(\rho_{r_2}(g)\psi_2-\psi_2)}\igap d\mu(g)
\label{equation:error2}\\
&\qquad{}+\int(\rho_{r_1}(g)\psi_1-\psi_1)
\cdot\overline{(\rho_{r_2}(g)\psi_2-\psi_2)}\igap d\mu(g)\label{equation:error3}.
\end{align}
To show the claim, we prove that
\[
|\langle S_{r_1-r_2}(\psi_1\overline{\psi_2}),\psi_1\overline{\psi_2}\rangle|\ge 1-4\e^2.
\]
Since $\langle\eqref{equation:mainterm},\psi_1\overline{\psi_2}\rangle=1$, it suffices to show that
\[
|\langle \eqref{equation:error1}+\eqref{equation:error2}+\eqref{equation:error3},
\psi_1\overline{\psi_2}\rangle|\le 8\e^2.
\]

We deal with the contributions of 
\eqref{equation:error1}--\eqref{equation:error3} separately.
To estimate the contribution of \eqref{equation:error1}, we write
\begin{align*}
|\langle\eqref{equation:error1},\psi_1\overline{\psi_2}\rangle|
&=\left|\int(\rho_{r_1}(g)\psi_1(\xi)-\psi_1(\xi))\cdot\overline{\psi_1(\xi)}\igap d\mu(g)d\xi\right|\\
&=|\langle S_{r_1}\psi_1-\psi_1,\psi_1\rangle|\le2 \e^2.
\end{align*}
The last inequality follows from $\e$-invariance and Lemma \ref{lemma:einvariance}.
An analogous inequality for the contribution of \eqref{equation:error2} follows from a similar argument.

To estimate \eqref{equation:error3}, we use the Cauchy-Schwartz inequality two times
and Lemma \ref{lemma:approximating2}:
\begin{align*}
\|\eqref{equation:error3}\|_1
&\le\int\|\rho_{r_1}(g)\psi_1-\psi_1\|_2
\cdot\|\rho_{r_2}(g)\psi_2-\psi_2\|_2\igap d\mu(g)\\
&\le\left(\int\|\rho_{r_1}(g)\psi_1-\psi_1\|_2^2\igap d\mu(g)
\cdot\int\|\rho_{r_2}(g)\psi_2-\psi_2\|_2^2\igap d\mu(g)\right)^{1/2}\\
&\le4\e^2.
\end{align*}
Using $\|\psi_1\overline{\psi_2}\|_\infty=1$, this yields
\[
|\langle\eqref{equation:error3},\psi_1\overline{\psi_2}\rangle|\le 4\e^2.
\]

Combining our estimates, we get the lemma.
\end{proof}

%%%%%%%%%%%%%%%%%%%%%%%%%%%%%%%%%%%%%%%%%%%%%
\subsection{Estimating \texorpdfstring{$\|S_r\|$}{the norm of Sr} near \texorpdfstring{$r=0$}{r=0}}
\label{section:Fr}
%%%%%%%%%%%%%%%%%%%%%%%%%%%%%%%%%%%%%%%%%%%%%

In this section, we estimate the norm of $S_r$ for small values of $r$.
This is done using the spectral gap property of $S_0$ and Taylor expansion.

\begin{lem}\label{lemma:smallr}
For every $r\le cM^{-1}$, we have $\|S_r\|\le1-cr^2$, where $c$ is a number depending only on the dimension $d$.
\end{lem}

In the proof we will need some estimates for the function
\be\label{equation:definitionF}
F_r:=S_r1=\int\omega_r(g)\igap d\mu(g).
\ee

\begin{lem}\label{lemma:Fr}
There is an absolute constant $C$,
such that
\[
\|1-F_r\|_2\le Cr^2.
\]
In addition, there is a constant $c$ depending only on the dimension $d$ such that
for every $0\le r\le cM^{-1}$ we have
\[
\|F_r\|_2\le 1-cr^2.
\]
\end{lem}

\begin{proof}
By Taylor's theorem,
\[
|1-2\pi ir\langle v(g),\xi\rangle-\omega_r(g)(\xi)|\le4\pi^2r^2|v(g)|^2.
\]
We integrate this and use assumptions \eqref{equation:centered} and  \eqref{equation:moment2}:
\[
|1-F_r(\xi)|
=\left|\int 1-2\pi ir\langle v(g),\xi\rangle-\omega_r(g)(\xi)\igap d\mu(g)\right|
\le 4\pi^2r^2,
\]
this proves the first claim.

For the second claim, we use one more term in the Taylor expansion of $\omega_r$ and integrate
it as above.
We get
\be\label{equation:FG}
\left|1-4\pi^2r^2\int\langle v(g),\xi\rangle^2\igap d\mu(g)-F_r(\xi)\right|\le8\pi^3Mr^3.
\ee

We consider the function
\[
G_r(\xi)=1-4\pi^2r^2\int\langle v(g),\xi\rangle^2\igap d\mu(g).
\]
If $r\le1/(2\pi)$ (that we may assume), then
\[
4\pi^2r^2\int\langle v(g),\xi\rangle^2\igap d\mu(g)\le 1,
\]
hence
\begin{align*}
\|G_r\|_1
&=\int_{S^{d-1}} G_r(\xi)\igap d(\xi)
=1-4\pi^2r^2\int\int\langle v(g),\xi\rangle^2\igap d\xi d\mu(g)\\
&=1-4\pi^2r^2c_d\int |v(g)|^2d\mu(g)
=1-4\pi^2r^2c_d,
\end{align*}
where $c_d$ is a number depending only on $d$.

Then we can write
\[
\|G_r\|_2\le(\|G_r\|_\infty\|G_r\|_1)^{1/2}\le1-2\pi^2r^2c_d.
\]
Combining with \eqref{equation:FG}, we get
\[
\|F_r\|_2\le1-2\pi^2r^2c_d+8\pi^3Mr^3.
\]
If $r\le c_d/(8\pi M)$ (that we may assume), then the second claim follows.
\end{proof}

\begin{proof}[Proof of Lemma \ref{lemma:smallr}]
Let $\f\in L^2(S^{d-1})$ be an arbitrary function of unit norm.
Write $A=\int\f(\xi) \igap d\xi$ and $\f=A+\f_0$
(for notational convenience, we assume as we may that $A\geq0$, in particular real).
Then by \eqref{equation:gapS0}, we have
\[
\|S_0\f_0\|_2\le\|\f_0\|_2/2.
\]
By Taylor's theorem,
\[
\|S_r\f_0-S_0\f_0\|_2\le\int\|\omega_r(g)-1\|_\infty\|\rho_0(g)\f_0\|_2\igap d\mu(g)
\le Cr\|\f_0\|_2.
\]
We can assume without loss of generality, that $r$ is sufficiently small
so that $\|S_r\f_0\|_2\le\|\f_0\|_2/\sqrt2$.

We can write
\[
\|S_r\f\|_2^2=\|AF_r+S_r\f_0\|_2^2
=A^2\|F_r\|_2^2+\|S_r\f_0\|_2^2+2\Re\langle AF_r,S_r\f_0\rangle.
\]
We write
\[
\langle F_r,S_r\f_0\rangle
=\langle S_rF_r,\f_0\rangle
=\langle S_{r}^21,\f_0\rangle
=-\langle 1-S_{r}^21,\f_0\rangle.
\]
Similarly to  Lemma \ref{lemma:Fr}, we can get analogous estimates for the function $S_r^21$.
Thus
$|\langle F_r,S_r\f_0\rangle|\le Cr^2\|\f_0\|_2$.

We plug this into the previous identity and use the inequality between the
arithmetic and geometric means:
\begin{align*}
\|S_r\f\|_2^2&\le A^2(1-cr^2)+(1-A^2)/2+Cr^2A\sqrt{1-A^2}\\
&\le A^2(1-cr^2/2)+(1-A^2)\left(\frac{1}{2}+\frac{C^2r^2}{2c}\right).
\end{align*}
It is easy to see that the right hand side takes its maximum for $A=1$ if $r$
is sufficiently small.
This proves the lemma.
\end{proof}

%%%%%%%%%%%%%%%%%%%%%%%%%%%%%%%%%%%%%%%%%%%%%
\subsection{Proof of Proposition \ref{proposition:preliminary}}
\label{section:proofPreliminary}
%%%%%%%%%%%%%%%%%%%%%%%%%%%%%%%%%%%%%%%%%%%%%

Suppose that $\f_1$ and $\f_2$ are two orthonormal $\e$-invariant functions for $S_r$ for some $\e,r\ge0$.
Write $\psi_i=\f_i/|\f_i|$.
By Lemma \ref{lemma:absolute}, these satisfy the assumptions in Lemma \ref{lemma:approximating}
with $r_1=r_2=r$.
Thus
\be\label{equation:lower}
\|S_0(\psi_1\overline{\psi_2})\|_2\ge 1-C\e^2.
\ee

Write
\begin{align*}
A&:=\left|\int_{S^{d-1}}\psi_1(\xi)\overline{\psi_2}(\xi)\igap d\xi\right|\\
&\le \left|\int_{S^{d-1}}\f_1(\xi)\overline{\f_2}(\xi)\igap d\xi\right|+\|\f_1-\psi_1\|_2+\|\f_2-\psi_2\|_2\\
&\le C\e
\end{align*}
using Lemma \ref{lemma:absolute}.
Therefore, by assumption  \eqref{equation:gapS0}, we have
\[
\|S_0(\psi_1\overline{\psi_2})\|_2^2\le A^2+\frac{1}{4}(1-A^2)\le \frac{1}{4}+C\e^2
\]
This combined with \eqref{equation:lower} yields that $\e>c$ for some absolute constant $c>0$.
This proves the first part of the proposition.

For the second part, let $r_1\ge r_2\ge 0$ such that
$r_1-r_2\le cM^{-1}$ with the constant $c$
from Lemma \ref{lemma:smallr}.
Let $\f_1,\f_2\in L^{2}(S^{d-1})$ be of unit norm and $\e$-invariant for $S_{r_1}$ and $S_{r_2}$, respectively.
Then by Lemmata \ref{lemma:absolute} and \ref{lemma:approximating},
\[
\|S_{r_1-r_2}(\psi_1\overline{\psi_2})\|_2\ge1-C\e^2.
\]
On the other hand by Lemma \ref{lemma:smallr}, we have $\|S_{r_1-r_2}\|\le1-c(r_1-r_2)^2$,
hence we must have $\e\ge c(r_1-r_2)$, which proves the second part of the proposition.

%%%%%%%%%%%%%%%%%%%%%%%%%%%%%%%%%%%%%%%%%%%%
\section{Non-concentration on subgroups}
\label{section:decay}
%%%%%%%%%%%%%%%%%%%%%%%%%%%%%%%%%%%%%%%%%%%%

Fix an arbitrary point $x_0\in\R^d$.
Write $\nu_l=\mu^{*(l)}.\d_{x_0}$.
This is the probability law of the $l$th step of the random walk starting from the point $x_0$.
In this section, we estimate the probability that the $l$th step is in a
fixed small ball.
This implies an estimate on the $\mu^{*(l)}$-measure of a neighborhood of
a subgroup of $\Isom(\R^d)$ conjugate to $\SO(d)$.
Denote by $B(r,x)\subset\R^d$ the ball of radius $r$ around a point $x\in\R^d$.

\begin{prp}\label{proposition:decay}
Suppose that \eqref{equation:symmetric}--\eqref{equation:moment3} hold.
Then there is a constant $C$ depending only on the dimension $d$, such that
the following holds.
Let $L>CM^2$ be a number.
 Then for every $1/2\ge r\ge 0$ and $l\ge L\log(r^{-1})$ and $y_0\in\R^d$, we have
\[
\nu_l(B(L^{1/2}r,y_0))\le C r^{\frac{d-1}{2(d+1)}}.
\]
\end{prp}

This proposition follows easily from the following estimate on the Fourier transform.

\begin{prp}\label{proposition:decay2}
Suppose that \eqref{equation:symmetric}--\eqref{equation:moment3} hold.
Then there is a constant $C$ depending only on the dimension $d$, such that
the following holds.
Let $L>CM^2$ be a number.
Then for every $R\ge 2$ and $l\ge L\log R$, we have
\[
(L^{-1/2}R)^{-(d-1)}\cdot\int_{|\xi|=L^{-1/2}R}|\wh\nu_l(\xi)|^2\igap d\xi
\le C R^{-\frac{d-1}{d+1}}.
\]
\end{prp}

The proof is based on the following simple lemma, which
provides a decomposition of an arbitrary probability measure into two parts, one whose Fourier transform
has small Lipschitz norm and one whose Fourier transform has small averages on large balls.
This implies that if the Fourier transform has large average on a sphere
then it must have large averages on spheres nearby, as well.
We will apply this to the measure $\nu_l*\check \nu_l$, to get a similar statement about $L^2$ averages
of $\wh\nu_l$.
We will compare this with the results of Section \ref{section:preliminary}, and conclude
that the $L^2$ averages of $\wh\nu_l$ on spheres must decay fast, as stated in Proposition \ref{proposition:decay2}.

\begin{lem}\label{lemma:decomposition}
Let $\eta$ be a probability measure on $\R^d$.
Then, for any $r\ge 0$, there are measures $\eta_1$
and $\eta_2$ such that
$\eta=\eta_1+\eta_2$,
\[
\|\wh\eta_1\|_{\Lip}\le Cr \quad{\rm and}\quad
\left|\int_{|\xi|=R}\wh\eta_2(\xi)\igap d\xi\right|\le C\left(\frac{R}{r}\right)^{(d-1)/2}.
\]
\end{lem}
\begin{proof}
Write $\eta_1=\eta\restr{B(0,r)}$.
Then for any unit vector $v\in\R^d$, we have
\[
\left|\frac{\partial\wh\eta_1(\xi)}{\partial v}\right|
\le\int_{B(0,r)}\left|\frac{\partial e(\langle x,\xi\rangle)}{\partial v}\right|\igap d\eta(x)
\le2\pi r.
\]
This immediately implies the first claim.

For the second claim, we write
\begin{align*}
\int_{|\xi|=R}\wh\eta_2(\xi)\igap d\xi
&=\int_{|\xi|=R}\int_{|x|>r}e(\langle x,\xi\rangle)\igap d\eta(x)d\xi\\
&= R^{d-1}\int_{|x|>r}\int_{|\xi|=1}e(\langle Rx,\xi\rangle)\igap d\xi d\eta(x).
\end{align*}
It is well-known (see e.g. \cite{Stein-Harmonic-analysis}*{Chapter VIII.6}, in particular formula (26) in that chapter)
that
\[
\left|\int_{|\xi|=1}e(\langle Rx,\xi\rangle)\igap d\xi\right|\le C|Rx|^{-(d-1)/2}
\]
with a number $C$ depending only on $d$.
This yields
\[
\left|\int_{|\xi|=R}\wh\eta_2(\xi)\igap d\xi\right|\le CR^{d-1}\cdot(Rr)^{-(d-1)/2},
\]
which was to be proved.
\end{proof}

\begin{proof}[Proof of Proposition \ref{proposition:decay2}]
Let $c_0$ be the constant $c$ from Proposition \ref{proposition:preliminary}.
Write $R_1=L^{-1/2}R-(c_0L/10)^{-1/2}$ and $R_2=L^{-1/2}R$.
Since $R\ge1$, we  have $R_1\ge R_2/2$.

By the assumption $L\ge CM^2$, we have $R_2-R_1\le c_0 M^{-1}$, if the constant $C$
is sufficiently large.
We apply Proposition \ref{proposition:preliminary} with $r_1=R_1$ and $r_2=R_2$.
If $\|S_{R_2}\|\le1-10L^{-1}$, then the proposition follows immediately
from the identity $\Res_{R_2}(\wh\nu_l)=S_{R_2}^l(\Res_{R_2}(\wh\d_{x_0}))$.
If this is not the case, then we have $\|S_{R_1}\|\le1-10L^{-1}$ by Proposition
\ref{proposition:preliminary}.

Now we fix $l\ge L\log R$.
Then
\be\label{equation:Rprime}
\int_{|\xi|=R_1} |\wh\nu_l(\xi)|^2\igap d\xi
\le C R_1^{d-1}(1-10L^{-1})^l
\le C R_2^{d-1}R^{-10}.
\ee
Recall that $R_2/2\le R_1\le R_2$.

Consider the measure $\eta=\nu_l*\check{\nu_l}$
(where as before $\check{\nu_l}$ is obtained from $\nu_l$ by reflection as in \eqref{eq_defcheck}).
Notice that \eqref{equation:Rprime} turns into
\be\label{equation:Rprime2}
\int_{|\xi|=R_1} \wh\eta(\xi)\igap d\xi\le C R_2^{d-1}R^{-10}.
\ee

We apply Lemma \ref{lemma:decomposition}
for $\eta$ with $r=L^{1/2}R^{-(d-1)/(d+1)}$.
Then we have
\begin{align*}
\left|\int_{|\xi|=R_1} \wh\eta_1(\xi)\igap d\xi\right|
&\le\left|\int_{|\xi|=R_1} \wh\eta(\xi)\igap d\xi\right|+\left|\int_{|\xi|=R_1} \wh\eta_2(\xi)\igap d\xi\right|\\
&\le C R_2^{d-1}R^{-10}+C\left(\frac{L^{-1/2}R}{L^{1/2}R^{-(d-1)/(d+1)}}\right)^{(d-1)/2}\\
&\le CR_2^{d-1}R^{-(d-1)/(d+1)}.
\end{align*}
Using the Lipschitz norm bound on $\wh\eta_1$ in the lemma, we get
\begin{align*}
\left|\int_{|\xi|=R_2} \wh\eta_1(\xi)\igap d\xi\right|
&\le\frac{R_2^{d-1}}{R_1^{d-1}}\left|\int_{|\xi|=R_1} \wh\eta_1(\xi)\igap d\xi\right|\\
&\qquad{}+C(R_2-R_1)L^{1/2}R^{-\frac{d-1}{d+1}}\cdot R_2^{d-1}\\
&\le CR_2^{d-1}R^{-\frac{d-1}{d+1}}.
\end{align*}
Finally, using the bound on $\eta_2$ in the lemma again, we get
\begin{align*}
\left|\int_{|\xi|=R_2}\wh\eta(\xi)\igap d\xi\right|
&=\left|\int_{|\xi|=R_2} \wh\eta_1(\xi)\igap d\xi\right|
+\left|\int_{|\xi|=R_2} \wh\eta_2(\xi)\igap d\xi\right|\\
&\le CR_2^{d-1}R^{-\frac{d-1}{d+1}}.
\end{align*}
This yields the claim upon substituting $\wh\eta=|\wh\nu_l|^2$.
\end{proof}

\begin{proof}[Proof of Proposition \ref{proposition:decay}]
Fix some $l\ge L\log(r^{-1})$.
Let $F:\R^d\to\R$ be a non-negative radially-symmetric function such that $F(x)\ge1$
for $|x|\le1$ and $\wh F$ is supported in the ball $B(1,0)$.
For $r\ge0$, write $F_{r,y_0}(x)=F((x-y_0)/r)$.
Then
\begin{align*}
\nu_l(B(L^{1/2}r,y_0))&\le\int F_{L^{1/2}r,y_0}(x)\igap d\nu_l(x)
=\int \wh F_{L^{1/2}r,y_0}(\xi)\wh\nu_l(\xi)\igap d\xi\\
&=L^{-1/2}\int_0^\infty\int_{|\xi|=L^{-1/2}R} \wh F_{L^{1/2}r,y_0}(\xi)\wh\nu_l(\xi)\igap d\xi dR\\
&\le L^{-1/2}\| F_{L^{1/2}r,y_0}\|_1\int_0^{r^{-1}}\int_{|\xi|=L^{-1/2}R} |\wh\nu_l(\xi)|\igap d\xi dR.
\end{align*}
Note that $ \wh F_{L^{1/2}r,y_0}$ is supported in $B(L^{-1/2}r^{-1},0)$.

By the Cauchy-Schwartz inequality and Proposition \ref{proposition:decay2}, we have
\begin{align*}
&(L^{-1/2}R)^{-(d-1)}\cdot\int_{|\xi|=L^{-1/2}R} |\wh\nu_l(\xi)|\igap d\xi\\
&\qquad{}\le C\left[(L^{-1/2}R)^{-(d-1)}\cdot\int_{|\xi|=L^{-1/2}R} |\wh\nu_l(\xi)|^2\igap d\xi\right]^{1/2}
\le C R^{-\frac{d-1}{2(d+1)}}.
\end{align*}
Strictly speaking, we proved this inequality only for $R\ge2$, however, it follows from
the trivial estimate $|\wh\nu_l(\xi)|\le1$ for $R\le 2$.

Notice that $\| F_{L^{1/2}r,y_0}\|_1=(L^{1/2}r)^{d}\|F\|_1$.
Combining these estimates, we get
\[
\nu_l(B(L^{1/2}r,y_0))\le (L^{1/2}r)^{d}\|F\|_1\cdot C L^{-d/2}r^{-\left(d-\frac{d-1}{2(d+1)}\right)},
\]
which was to be proved.
\end{proof}

%%%%%%%%%%%%%%%%%%%%%%%%%%%%%%%%%%%%%%%%%
\section{Sets of large dimension in compact Lie groups}
\label{section:compactgroups}
%%%%%%%%%%%%%%%%%%%%%%%%%%%%%%%%%%%%%%%%%

In this section, we examine sets and measures of ``large dimension" in compact groups. Here we use the word dimension somewhat loosely,
and only to illuminate the results by an informal
interpretation.
For our purposes, a set of ``large dimension" at scale $r$ is a set which contains at least
$r^{-a}$ disjoint balls of radius $r$, where $a$ is ``large", depending on the situation.
A measure of ``large dimension" at scale $r$ is one which puts at most $r^a$ mass on
a ball of radius $r$ with $a$ ``large".

We prove variants of results of Bourgain and Gamburd \cite{BG-SUd} and Saxc\'e \cite{Sax-largedim}.
We follow the treatment of Saxc\'e based on exploiting high multiplicities of irreducible components
in the regular representation.
This idea goes back to Sarnak and Xue \cite{SX-multiplicities} in a different setting.
For alternative treatments, see \cite{BG-SUd} and \cite{GJS-SU2} by Gamburd, Jakobson and Sarnak.

Throughout the section, let $G$ be a compact Hausdorff topological group.
We denote the Haar measure on $G$ by $m$, normalized to have total mass $1$.
Let $\pi$ be a unitary representation of $G$.
Recall \eqref{equation:piofmu}, the definition of $\pi(\mu)$.
If $f\in L^1(G)$, then we consider it as the density of a measure on $G$ and
define the operator $\pi(f)$ similarly to \eqref{equation:piofmu}.
When $\pi$ is a unitary representation of $G$, $\pi(f)$ is the analogue of
the Fourier coefficients in the theory of functions on $\R/\Z$.

We first present a corollary of Schur's Lemma.
Bourgain and Gamburd \cite{BG-SUd} exploited a variant of this result in their method
to establish norm estimates for operators related to random walks.

\begin{prp}\label{proposition:BG-largedim}
With notation as above, let $\pi$ be a unitary representation of $G$, and let $D$ be a number
such that all irreducible components of $\pi$ are of dimension at least $D$.
Then for any vectors $u,v$ in the representation space of $\pi$, we have
\[
\left[\int|\langle \pi(g)u,v \rangle|^2 \igap dm(g)\right]^{1/2}\le\frac{\|u\|\|v\|}{D^{1/2}}.
\]
\end{prp}

We illustrate the purpose of Proposition \ref{proposition:BG-largedim} by sketching how it
can be used  to estimate the norm of $\pi(\mu)$ for a measure of ``large
dimension".
To this end, we can approximate $\mu$ with a measure with bounded density $f$ and write
\[
|\langle\pi(f)u,v\rangle|\le\int \|f\|_\infty\cdot|\langle\pi(g)u,v\rangle|\igap dm(g).
\]
If $\mu$ is of ``large dimension" at some scale $r$, and $\pi$ is ``not sensitive" to perturbations
at this scale and $D$ is ``large" compared to $r^{-1}$, then the above bound combined with the
proposition is non-trivial.
The proposition will be used to obtain similar results for the non-compact
group $\Isom(\R^d)$, see Propostion \ref{proposition:largedimension}.
Then we will execute an argument similar to the above sketch (cf. Section \ref{section:completing}).

The following result is due to Saxc\'e \cite{Sax-largedim}*{Proposition 4.5}.
It allows us to find a large open ball in the product set of three sets of large dimension.
This is an analogue of results of Gowers
\cite{Gow-quasirandom} and Nikolov and Pyber \cite{NP-product} in finite groups.

\begin{prp}\label{proposition:Saxce}
Let $G$ be a connected semisimple compact Lie group endowed with a probability Haar measure $m$.
There is a constant $C$ depending on the group $G$ such that the following holds.
Let $A_1,A_2,A_3\subset G$ be Borel subsets.
Then the set $A_1A_2A_3$ contains an open ball of radius at least
\[
\frac{1}{C}(m(A_1)m(A_2)m(A_3))^C.
\]
\end{prp}

We note that Saxc\'e's formulation of this result does not estimate the radius of the ball, which is crucial
for our application (on the other hand Saxc\'e's statement deals with sets of large Hausdorff dimension but zero Haar measure), therefore we reproduce the result with essentially the same proof.

The rest of this section is devoted to the proof of Propositions~\ref{proposition:BG-largedim} and~\ref{proposition:Saxce}.
We denote the set of irreducible unitary representations of $G$ (up to isomorphism) by $\wh G$.
By the theorem of Peter and Weyl, these are all finite dimensional, and if $f\in L^2(G)$,
we have the analogue of Plancherel's formula:
\be{\label{equation:plancherel}}
\|f\|_2^2=\sum_{\pi\in\wh G}\dim \pi\|\pi(f)\|_\HS^2,
\ee
where $\|X\|_\HS=\Tr(X^*X)^{1/2}$ is the Hilbert-Schmidt norm of the operator $X$ (c.f. \cite[Sect. I.5]{Knapp-representations}).
Moreover, we have the Fourier inversion formula
\be\label{equation:Finversion}
f(g)=\sum_{\pi\in\wh G}\dim \pi \Tr(\pi(g)\pi(f)).
\ee

\begin{lem}\label{lemma:Fourierbound}
With notation as above,
let $f\in L^2(G)$ and $\pi$ be an irreducible unitary representation of $G$.
Then
\[
\|\pi(f)\|_\HS\le\frac{\|f\|_2}{\sqrt{\dim \pi}}.
\]
\end{lem}
\begin{proof}
This follows from \eqref{equation:plancherel}, since all terms are non-negative on the right hand side.
\end{proof}

\begin{proof}[Proof of Proposition \ref{proposition:BG-largedim}]
We first note that when $\pi$ is irreducible then the statement is contained in Schur's lemma,
see e.g. \cite[Corollary 1.10 (b)]{Knapp-representations}.

If $\pi$ is not irreducible, then 
we decompose it as the sum of irreducible components $\pi_1\oplus\ldots\oplus\pi_n$,
and write $u=u_1+\ldots + u_n$ and $v=v_1+\ldots +v_n$, where $u_i$ and $v_i$
are the components of $u$ and $v$ in the space of $\pi_i$.

We can write using Minkowski's inequality, the irreducible case and the Cauchy-Schwartz inequality
\begin{align*}
\left[\int|\langle \pi(g)u,v \rangle|^2\igap dm(g)\right]^{1/2}
&=\left[\int\left|\sum_{i=1}^n\langle \pi_i(g)u_i,v_i \rangle\right|^2\igap dm(g)\right]^{1/2}\\
&\le\sum_{i=1}^n\left[\int|\langle \pi_i(g)u_i,v_i \rangle|^2\igap dm(g)\right]^{1/2}\\
&\le\sum_{i=1}^n\frac{\|u_i\|\|v_i\|}{D^{1/2}}\\
&\le\frac{(\sum_{i=1}^n\|u_i\|^2)^{1/2}(\sum_{i=1}^n\|v_i\|^2)^{1/2}}{D^{1/2}}\\
&=\frac{\|u\|\|v\|}{D^{1/2}}.
\end{align*}
\end{proof}

We turn to the proof of Proposition \ref{proposition:Saxce}.
This requires some basic information about the representation theory of semisimple
Lie groups.
The irreducible unitary representations of a semisimple compact Lie group $G$ can be parametrized by
integer vectors $v$ called highest weights.
We denote by $\pi_v\in\wh G$ the irreducible representation with highest weight $v$ and note that
by Weyl's dimension formula \cite[Thm. 4.48]{Knapp-representations}, we have the bounds $|v|^{a}\le\dim\pi_v\le |v|^b$ with some constants 
$a,b$ depending only on $G$.

We also need to bound the Lipschitz norm of a function contained in representations of small
highest weights.
Let $r>0$ be a number and write $\cH_r<L^2(G)$ for the sum of the irreducible components with highest
weight $|v|\le r$ in the regular representation of $G$.
We recall the following simple estimate from \cite{Var-compact}.

\begin{lem}[{\cite{Var-compact}*{Lemma 20}}]\label{lemma:Lipschitz}
For any semisimple compact Lie group, there is a constant $C$ such that
$\|f\|_\Lip\le Cr^C\|f\|_2$ for any functions $f\in\cH_r$
\end{lem}

\begin{proof}[Proof of Proposition \ref{proposition:Saxce}]
We write $f_i$ for the indicator function of $A_i$.
We estimate $\Tr(\pi(g)\pi(f_1*f_2*f_3))$ in terms of the Hilbert-Schmidt norm of the
Fourier coefficients $\pi(f_i)$.
If $X=(X_{i,j}),Y=(Y_{i,j})$ are any square matrices, then by the Cauchy-Schwartz inequality, we have
\begin{align*}
|\Tr(X^*Y)|&=\left|\sum_i\sum_k\overline{X_{k,i}}Y_{k,i}\right|\\
&\le\left(\sum_{k,i}|X_{k,i}|^2\right)^{1/2}\left(\sum_{k,i}|Y_{k,i}|^2\right)^{1/2}
=\|X\|_{\HS}\|Y\|_\HS.
\end{align*}
We use this with $X=\pi_v(g)\pi_v(f_1)$ and $Y=\pi_v(f_2)\pi_v(f_3)$ and get
\begin{align*}
|\Tr(\pi_v(g)\pi_v(f_1*f_2*f_3))|&\le\|\pi_v(g)\pi_v(f_1)\|_\HS\|\pi_v(f_2)\pi_v(f_3)\|_\HS\\
&\le\|\pi_v(f_1)\|_\HS\|\pi_v(f_2)\|_\HS\|\pi_v(f_3)\|_\HS.
\end{align*}

We fix a number $r>0$ to be specified later,
and write $f_1*f_2*f_3=\f_0+\f_1$, where $\f_1$
is the tail of the series \eqref{equation:Finversion}:
\begin{align*}
|\f_1(g)|&:=\left|\sum_{|v|>r}\dim{\pi_v}\Tr(\pi_v(g)\pi_v(f_1*f_2*f_3))\right|\\
&\le\sum_{|v|>r}\dim{\pi_v}\|\pi_v(f_1)\|_\HS\|\pi_v(f_2)\|_\HS\|\pi_v(f_3)\|_\HS.
\end{align*}
We use Lemma \ref{lemma:Fourierbound} together with the bound $\dim\pi_v\ge r^a$
to estimate $\|\pi_v(f_1)\|_\HS$
and then use the Cauchy-Schwartz inequality and Plancherel's formula \eqref{equation:plancherel}:
\begin{align}
|\f_1(g)|&\le
\sum_{|v|>r}\dim{\pi_v}r^{-a/2}\|f_1\|_2\|\pi_v(f_2)\|_\HS\|\pi_v(f_3)\|_\HS\nonumber\\
&\le r^{-a/2}\|f_1\|_2\left(\sum_{|v|>r}\dim{\pi_v}\|\pi_v(f_2)\|_\HS^2\right)^{1/2}\nonumber\\
&\qquad{}\times\left(\sum_{|v|>r}\dim{\pi_v}\|\pi_v(f_3)\|_\HS^2\right)^{1/2}\nonumber\\
&\le r^{-a/2}\|f_1\|_2\|f_2\|_2\|f_3\|_2.\label{equation:supestimate}
\end{align}

Since $m$ is a probability measure, there is a point $g_0\in G$ such that
\be\label{equation:g0}
|f_1*f_2*f_3(g_0)|\ge\|f_1*f_2*f_3\|_1=m(A_1)m(A_2)m(A_3).
\ee
We fix a number $\rho>0$ to be specified later and prove that $|f_1*f_2*f_3(g)|>0$  if
$\dist(g,g_0)\le\rho$.
By Lemma \ref{lemma:Lipschitz}, we have
\begin{align*}
|f_1*f_2*f_3(g_0)-f_1&*f_2*f_3(g)|
\le|\f_0(g)-\f_0(g_0)|+|\f_1(g)-\f_1(g_0)|\\
&\le|\f_1(g)|+|\f_1(g_0)|+Cr^C\rho\|f_1*f_2*f_3\|_2.
\end{align*}
We combine this with \eqref{equation:supestimate} and \eqref{equation:g0} and use the trivial estimates $\|f_i\|_2\le1$:
\[
|f_1*f_2*f_3(g)|\ge m(A_1)m(A_2)m(A_3)-2r^{-a/2}-Cr^C\rho.
\]

We now take
\[
r=(m(A_1)m(A_2)m(A_3)/10)^{-2/a}\quad{\rm and}\quad
\rho=r^{-C-a/2}/10C
\]
and conclude the proof.
\end{proof}

%%%%%%%%%%%%%%%%%%%%%%%%%%%%%%%%%%%%%%%%%
\section{The Bourgain--Gamburd method: flattening}
\label{section:bourgaingamburd1}
%%%%%%%%%%%%%%%%%%%%%%%%%%%%%%%%%%%%%%%%%

In this and the next section, we recall the Bourgain--Gamburd meth\-od
and adapt it to prove Theorem \ref{theorem:technical}.
The method has been developed in
\cite{BG-prime}, \cite{BG-SU2} and several subsequent papers.
In these sections, we heavily rely on the ideas of Bourgain and Gamburd but there are
a few new ingredients, most notably Lemma \ref{lemma:translation}.

We show in this section that if we convolve the distribution of the random
walk with itself (that is, we double the number of steps), then we obtain
a measure with better non-concentration properties.
In the next section, we iterate this and obtain nearly optimal non-concentration bounds
and use them to deduce the bounds on $\|S_r\|$ claimed in Theorem \ref{theorem:technical}.

To formalize this, we introduce some notation.
Let $1>\d>0$ be a number and $l\ge 1$ an integer.
We associate a neighborhood of $1\in\Isom(\R^d)$
to these parameters:
\[
B_{\d,l}:=\{(v,\t)\in\Isom(\R^d):\dist(\t,1)\le\d\;{\rm and}\;|v|\le\d\cdot l^{1/2}\}.
\]
To obtain an approximation at scale $\d$ with $L^2$ density,
we will convolve the random walk with the function
\[
P_{\d,l}(g):=\left\{
\begin{array}{cl}
\frac{1}{m(B_{\d,l})}& {\rm if\;} g\in B_{\d,l}\\
0&{\rm otherwise.}
\end{array}
\right.
\]

This section is devoted to the proof of the following proposition.

\begin{prp}\label{proposition:flattening2}
For any integer $d\ge3$ and $a>0$, there are $\a,C_0$ such that the following holds.
Let $\mu$ be a probability measure on $\Isom(\R^d)$ satisfying
\eqref{equation:symmetric}--\eqref{equation:moment3}.
Fix a number $1/2>\d>0$ and let $l_1\ge C_0M^2\log \d^{-1}$ be an integer.
Let $\eta=P_{\d,l_1}*\mu^{*(l_1)}*P_{\d,l_1}$.
Then for each integer $k\ge1$ we have either
\be\label{equation:flatten}
\|\eta^{*(2k)}\|_2\le \d^{\a}\cdot\|\eta^{*(k)}\|_2
\ee
or
\be\label{equation:flat}
\|\eta^{*(k)}\|_2\le C_0\d^{-a} l_1^{-d/4}.
\ee
\end{prp}

Upon iterating the proposition, we obtain
\[
\|(P_{\d,l_1}*\mu^{*(l_1)}*P_{\d,l_1})^{*(2^{k+1})}\|_\infty
\le\|(P_{\d,l_1}*\mu^{*(l_1)}*P_{\d,l_1})^{*(2^k)}\|_2^2
\le C \d^{-2a} l_1^{-d/2}
\]
for an arbitrarily small number $a>0$ if $k$ is suitably large.
It is crucial that $k$, the number of iterations we need to take, is independent of $\d$.
One can interpret this inequality as a very strong non-concentration of the random walk
on balls of radius $\d$.
Alternatively, we can also say with the terminology of Section~\ref{section:compactgroups}
that $\mu^{*(2^kl_1)}$
is of ``large dimension" at scale $\d$.

For the proof of the proposition, we can assume that $\d<\d_0$ for any constant $\d_0$ depending
on $d$ and $\a$ only.
Indeed, \[\|P_{\d,l_1}\|_2\le C \d^{-\dim\Isom(\R^d)/2}l_1^{-d/4},\] hence conclusion \eqref{equation:flat} holds for all $k$
if $\d\ge\d_0$ and $C_0$ is sufficiently large.

%%%%%%%%%%%%%%%%%%%%%%%%%%%%%%%%%%%%%%%%%
\subsection{Flattening}
\label{section:flattening}
%%%%%%%%%%%%%%%%%%%%%%%%%%%%%%%%%%%%%%%%%

We recall a useful result related to  the Balog Szemer\'edi Gowers theorem.
Let $G$ be a unimodular second countable locally compact  Hausdorff 
topological group endowed with
a bi-invariant Haar measure $m$.
Let $f\in L^2(G)$ be the density of a probability measure and let $A\subset G$.
Then the Cauchy-Schwartz inequality implies that
\[
\|f*f\|_2\ge m(A.A)^{-1/2}\int_{A.A}f*f\igap dm\ge m(A.A)^{-1/2}\left(\int_{A}f\igap dm\right)^2.
\]
Suppose that for some number $K$, we have
\[
\int_{A}f\igap dm>1/K,\quad m(A.A)\le K m(A), \quad \|f\|_2^{-2}/K\le m(A)\le K \|f\|_2^{-2},
\]
that is $f$ is concentrated on a set of small doubling of size comparable to $\|f\|_2^{-2}$.
Then the above inequality implies that $\|f*f\|_2\ge K^{-3}\|f\|_2$ that is the $L^2$ norm
is not decreased by convolution.

Luckily, there is a converse to this observation, which can be stated informally as follows:
If the $L^2$ norm is not decreased by convolution, then the function must necessarily concentrate
on a set of small tripling.
The exact formulation is contained in the next proposition.
The reason why we are looking for sets of small tripling as opposed to doubling is that
the quantity $m(A.A.A)/m(A)$ can be used to control the size of product sets of more factors,
whereas $m(A.A)/m(A)$ is not sufficient in general in non-commutative groups.

\begin{prp}\label{proposition:flattening}
There is an absolute constant $C$ such that the following holds.
Let $f_1,f_2\in L^{2}(G)$ be densities of probability measures, that is
$f_1,f_2\ge0$ and $\int f_1=\int f_2=1$.
Suppose that $\|f_1\|_2\ge\|f_2\|_2$ and $\|f_1*f_2\|_2\ge\|f_1\|_2/K$ for some number $K\ge1$.
Then there is a symmetric $F_\s$ set $A\subset G$ such that the following hold
\begin{align}
C^{-1}K^{-C}\|f_1\|_2^{-2}&\le m(A)\le CK^{C}\|f_1\|_2^{-2}, \label{equation:sizeA}\\
C^{-1}K^{-C}\|f_1\|_2^2&\le\check{f_1}*f_1(x)\quad
\text{for all $x\in A$},\label{equation:fofA}\\
m(A.A.A)&\le CK^Cm(A).
\end{align}
\end{prp}

The idea of this proposition goes back to the papers \cite{BG-SUd}, \cite{BG-prime} and it is an application
of the Balog Szemer\'edi Gowers theorem.
A discrete version of the present formulation can be found in \cite{Var-squarefree}*{Lemma 15}.
The proof given there can be adapted to the continuous setting in a straightforward manner.
The proof in the continuous setting is given in the forthcoming book \cite{LV-book}.

%%%%%%%%%%%%%%%%%%%%%%%%%%%%%%%%%%%%%%%%%
\subsection{Non-concentration on sets of small tripling}
\label{section:tripling}
%%%%%%%%%%%%%%%%%%%%%%%%%%%%%%%%%%%%%%%%%

We prove in this section a non-concentration estimate on sets of small tripling
and use Proposition \ref{proposition:flattening} to prove Proposition \ref{proposition:flattening2}.
The key properties used are the spectral gap of the projection to $\SO(d)$ and the fact
established in Sections \ref{section:preliminary} and \ref{section:decay}
that the random walk does not concentrate on a subgroup isomorphic to $\SO(d)$.

We suppose that the assumptions in
Proposition \ref{proposition:flattening2} hold for some
$a,d,\mu,l,\d$ and yet both conclusions \eqref{equation:flatten} and \eqref{equation:flat}
fail with certain numbers $\a,C_0$.
We derive a contradiction, if $\a$ is sufficiently small and $C_0$ is sufficiently large depending
only on $d$ and $a$.
The letters $c,C$ appearing below until the end of the section denote positive numbers that
depend on $d$ only, in particular they are independent of $a, \a$ and $\d$.
We will prove an inequality, which can not hold if $\a$ is chosen sufficiently small depending on
$d, a$ and the quantities denoted by $c,C$.
The argument will be valid if $\d$ is sufficiently small depending on $a,d,\a$ and the quantities denoted by $c,C$.
After we specified the values of all other parameters,
we set $C_0$ in the statement of Proposition \ref{proposition:flattening2}
to ensure that it is vacuous when $\d$ is not sufficiently small.

In what follows, we denote by $m$ the Haar measures on both $\Isom(\R^d)$ and $\SO(d)$.
On the first group we take an arbitrary normalization, on the second one we take it to be
a probability measure.

We apply Proposition \ref{proposition:flattening} for $K=\d^{-\a}$ and $f_1=f_2=\eta^{*(k)}$.
Then we get an $F_\s$-set $A\subset \Isom(\R^d)$ such that
\be\label{equation:A1}
m(A.A.A)\le C K^C m(A).
\ee

Moreover, combining equations \eqref{equation:sizeA} and \eqref{equation:fofA},
we get that
\[
\check\eta^{*(k)}*\eta^{*(k)}(x)\ge cK^{-C} m(A)^{-1}
\]
for $x\in A$.
We integrate this on $A$ and get
\be\label{equation:muA}
\int_A\check\eta^{*(k)}*\eta^{*(k)}(x)\igap dm(x)\ge c K^{-C}.
\ee

Finally, we add that \eqref{equation:sizeA} and the failure of \eqref{equation:flat},
e.g.\ for $C_0=1$,
yield
\[
m(A)\le CK^C\d^{2a}l_1^{d/2}.
\]
By \cite{Tao-noncommutative}*{Lemma 3.4} we then have
\be\label{equation:A58}
m(\textstyle\prod_{58} A)\le CK^C\d^{2a}l_1^{d/2}.
\ee
In order to get a contradiction, we proceed by a series of Lemmata giving
more and more information about larger and larger product sets of $A$.

\begin{lem}\label{lemma:neighborhood}
Let $A$ satisfy \eqref{equation:A1}--\eqref{equation:A58}.
Then the set $\t(\prod_6 A)$ contains an open ball of radius at least $cK^{-C}$
around $1\in \SO(d)$.
\end{lem}

\begin{proof}
Recall the operator $T$ acting on $L^2(\SO(d))$ by
\[
Tf(\s):=\int f(\t(g)^{-1}\s)\igap d\mu(g).
\]
By assumption \eqref{equation:gap},
\[
\|Tf\|_2\le \frac{1}{2}\|f\|_2
\]
for every function $f\in L^2(\SO(d))$ that satisfy $\int f\igap dm=0$.

Consider the functions $F_j\in L^2(\SO(d))$ for integers $j\ge0$
given by
\[
F_j(\t):=\int_{\R^d}\mu^{*(j)}*P_{\d,l_1}(v,\t)\igap dm(v).
\]
i.e.~$F_j$ is the density of the measure obtained by projecting
$\mu^{*(j)}*P_{\d,l_1}(v,\t)$ to $SO(d)$.
Observe that $F_j=T^jF_0$.
We can write $F_0=1+F'$ such that $\int F'(g)\igap dm(g)=0$ and $\|F'\|_2\le C\d^{-C}$,
where $C>0$ is a number depending only on $d$.
Recall that $l_1\ge C_0M^2\log(\d^{-1})$, where $C_0$ can be chosen suitably large depending on
$d$, and hence
\[
\|F_{l_1}\|_2\le1+\frac{1}{2^{l_1}}\cdot  C\d^{-C}\le2.
\]

This in turn yields
\[
\left\|\left(\int_{\R^d}\check\eta^{*(k)}*\eta^{*(k-1)}*P_{\d,l_1}(v,\t)\igap dm(v)\right)*F_{l_1}\right\|_2\le2
\]
hence\[
\int_{\SO(d)}\left[\int_{\R^d}\check\eta^{*(k)}*\eta^{*(k)}(v,\t)
\igap dm(v)\right]^2\igap dm(\s)
\le 4.
\]

By the Cauchy-Schwartz inequality:
\begin{align*}
&\int_{\t(A)}\left[\int_{\R^d}\check\eta^{*(k)}*\eta^{*(k)}(v,\t)\igap dm(v)\right]\igap dm(\s)\\
&\quad{}\le (m(\t(A)))^{1/2}\cdot
\left(\int_{\t(A)}\left[\int_{\R^d}\check\eta^{*(k)}*\eta^{*(k)}(v,\t)\igap dm(v)\right]^2
\igap dm(\s)\right)^{1/2}\\
&\quad{}\le 2 (m(\t(A)))^{1/2}.
\end{align*}
On the other hand
\begin{align*}
&\int_{\t(A)}\left[\int_{\R^d}\check\eta^{*(k)}*\eta^{*(k)}(v,\t)\igap dm(v)\right]\igap dm(\s)\\
&\qquad{}\ge \int_{A}\check\eta^{*(k)}*\eta^{*(k)}(g)\igap dm(g)\ge cK^{-C}.
\end{align*}
Combining the last two inequalities, we get
\[
m(\t(A))\ge c K^{-C}.
\]

By Proposition \ref{proposition:Saxce}, $\t(A\cdot A\cdot A)$ contains a ball of radius $cK^{-C}$.
Since $A$ is symmetric, $\t(\prod_6A)$ must contain such a ball centered at $1$.
\end{proof}

\begin{lem}\label{lemma:translation}
Let $A$ satisfy \eqref{equation:A1}--\eqref{equation:A58}.
Then the set $\prod_{14}A$
contains a pure translation of length at least $cK^{-C}\d^{a/d}l_1^{1/2}$.
\end{lem}

This lemma depends on the results of Sections \ref{section:preliminary} and \ref{section:decay}.
The information we need is contained in the next lemma.

\begin{lem}\label{lemma:decay}
Let $x_0,y_0\in\R^d$ be two points, $\d<s<1/4$ 
and denote by $\Omega\subset\Isom(\R^d)$ the set of isometries $g$
which satisfy
\[
g(x_0)\in B(l_1^{1/2}s,y_0).
\]
Then
\be\label{equation:decay}
\int_\Omega \check\eta^{*(k)}*\eta^{*(k)}(g) \igap dm(g)\le Cs^{\frac{d-1}{4(d+1)}}.
\ee
\end{lem}
\begin{proof}
Let $X_1,X_2,X_3\in\Isom(\R^d)$ be independent random isometries with laws
$\check\eta^{*(k)}*\eta^{*(k-1)}*P_{\d,l_1}$,
$\mu^{*(l_1)}$, $P_{\d,l_1}$ respectively.
Then the quantity on the left hand side of \eqref{equation:decay}
is the probability of the event that
\[
X_1X_2X_3x_0\in B(l_1^{1/2}s, y_0).
\]
This is equivalent to
\[
X_2(X_3x_0)\in B(l_1^{1/2}s, X_1^{-1}y_0).
\]
The probability of this is bounded by
\[
\max_{x,y\in\R^d}\P[X_2x\in B(l_1^{1/2}s,y)].
\]

We apply  Proposition \ref{proposition:decay} with $r=(\log s^{-1})^{1/2}s$
and $L=l_1/\log(r^{-1})$.
We note that $s\le r\le s^{1/2}$ (as $s\le 1$) and 
\[
L\ge \frac{l_1}{\log \d^{-1}}\ge CM^2,
\]
where $C$ can be any number if $C_0$ in Proposition \ref{proposition:flattening2}
is sufficiently large.
In particular, we can ensure that $C$ is so large that  Proposition \ref{proposition:decay} holds.
We get from the proposition that
\[
\max_{x,y\in\R^d}\P[X_2x\in B(L^{1/2}r,y)]\le C r^\frac{d-1}{2(d+1)}
\le C s^\frac{d-1}{4(d+1)}.
\]
To finish the proof, we observe that
\[
L^{1/2}r=\left(\frac{l_1}{\log(r^{-1})}\right)^{1/2}\cdot(\log s^{-1})^{1/2}s\ge l_1^{1/2}s.
\]
\end{proof}

\begin{proof}[Proof of Lemma \ref{lemma:translation}]
Denote by $\Theta\subset\SO(d)$ the $cK^{-C}$ neighborhood
of $1\in\SO(d)$.
By Lemma \ref{lemma:neighborhood}, we have $\Theta\subset\t(\prod_6A)$, hence
there is a measurable function $F:\Theta\to\prod_6A$ such that
$\t(F(\s))=\s$ for all $\s\in\Theta$.
Note that the set $\Theta$ is invariant under conjugation.

We look at isometries of the form
\[
gF(\s)g^{-1}F(\t(g)\s\t(g)^{-1})^{-1}
\]
for $g\in A$ and $\s\in\Theta$.
These are all pure translations, and we will see that their lengths
are not small for typical choices of $g$ and $\s$.

Denote by $m_\Theta$ the restriction of the Haar measure of $\SO(d)$
to $\Theta$ normalized to have total mass 1.
Write
\[
u_0=\int_\Theta v(F(\s))\igap dm_\Theta(\s).
\]

We choose an arbitrary $g\in A$ and recall that $v(g^{-1})=-\t(g)^{-1}v(g)$.
Then
\begin{align*}
\int_\Theta v(gF(\s)g^{-1}&F(\t(g)\s\t(g)^{-1})^{-1})\igap dm_\Theta(\s)\\
&=\int_\Theta v(g)+\t(g)v(F(\s))-\t(g)\s\t(g)^{-1}v(g)\\
&\qquad{}-v(F(\t(g)\s\t(g)^{-1}))\igap dm_\Theta(\s).
\end{align*}

A simple computation shows that there is a number $b$ such that
\[
\int_\Theta \s v\igap dm_\Theta(\s)=b\cdot v
\]
for any $v\in\R^d$.
Moreover, we have
\[
1-CK^{-C}\le b\le 1-cK^{-C}.
\]

Thus
\begin{align*}
&\int_\Theta v(gF(\s)g^{-1}F(\t(g)\s\t(g)^{-1})^{-1})\igap dm_\Theta(\s)\\
&\qquad{}=v(g)+\t(g)u_0-b\cdot v(g)-u_0\\
&\qquad{}=(1-b)[g((1-b)^{-1}\cdot u_0)-(1-b)^{-1}\cdot u_0].
\end{align*}

We apply Lemma  \ref{lemma:decay} for $x_0=y_0=(1-b)^{-1}\cdot u_0$
and $s=\d^{a/d}$.
We assume, as we may, that $\a$ and $\d$ are sufficiently small
(depending on $d,a$ and the $C$ below), so that
\[
\eqref{equation:muA}\ge Cs^{\frac{d-1}{4(d+1)}}.
\]
Then there is an isometry $g\in A$
such that 
\[
(1-b)|g((1-b)^{-1}\cdot u_0)-(1-b)^{-1}\cdot u_0|>cK^{-C}\d^{a/d}l_1^{1/2}.
\]
Hence there is $\s\in\Theta$ such that
\[
|v(gF(\s)g^{-1}F(\t(g)\s\t(g)^{-1})^{-1})|>cK^{-C}\d^{a/d}l_1^{1/2}.
\]
and this proves the lemma.
\end{proof}

\begin{proof}[Proof of Proposition \ref{proposition:flattening2}]
By Lemma \ref{lemma:translation}, there is $g_0\in\prod_{14} A$ which is
a pure translation of length at least $cK^{-C}\d^{a/d}l_1^{1/2}$.
The set
\[
\{h^{-1}g_0h: h\in\textstyle\prod_{6} A\}\subset\textstyle\prod_{26} A
\]
consists of pure translations and its difference set, which is a subset of $A^{52}$,
contains a ball of radius $cK^{-C}\d^{a/d}l_1^{1/2}$ in $\R^d$
by Lemma \ref{lemma:neighborhood}.
Using Lemma \ref{lemma:neighborhood} again, we get that there is a ball $\Theta\subset\SO(d)$
of radius $cK^{-C}$ such that for all $\s\in\Theta$,
\[
\t^{-1}(\s)\cap(\textstyle\prod_{58} A)
\]
contains a ball of radius $cK^{-C}\d^{a/d}l_1^{1/2}$ on the fiber.
Thus
\[
m(\textstyle\prod_{58}{A})\ge cK^{-C}\d^{a}l_1^{d/2},
\]
which contradicts \eqref{equation:A58} if $\a$ and $\d$ are sufficiently small.
\end{proof}

%%%%%%%%%%%%%%%%%%%%%%%%%%%%%%%%%%%%%%%%%
\section{The Bourgain--Gamburd method: norm estimates for measures of large dimension}
\label{section:bourgaingamburd2}
%%%%%%%%%%%%%%%%%%%%%%%%%%%%%%%%%%%%%%%%%

Recall our standing assumptions \eqref{equation:symmetric}--\eqref{equation:moment3}.
Fix a number $1/2>\d>0$.
Recall from the previous section the definition of $P_{\d,l_1}$ and that
$\eta=P_{\d,l_1}*\mu^{*(l_1)}*P_{\d,l_1}$
for some integer $l_1\ge CM^2\log\d^{-1}$.
In the previous section, we proved that there is an integer $K$ depending only on $d$
such that
\be\label{equation:strongNC}
\|\eta^{*(K)}\|_2\le C\d^{-1/20}l_1^{-d/4}.
\ee

Indeed, by the definitions of $P_{\d,l_1}$ and $\eta$ (cf.\ Section~\ref{section:tripling})
\[
\|\eta\|_2\le\|P_{\d,l_1}\|_2\le C\d^{-\dim\Isom(\R^d)/2}l_1^{-d/4}.
\]
Repeated applications of Proposition \ref{proposition:flattening2} with $a=1/20$ then gives:
\[
\|\eta^{*(2^k)}\|_2\le C\d^{\a k-\dim\Isom(\R^d)/2}l_1^{-d/4}
\]
so long as
\[
\|\eta^{*(2^{k-1})}\|_2\ge C\d^{-1/20}l_1^{-d/4}.
\]
It follows then, that there is an integer $K$ depending only on $d$ such that
\[
\|\eta^{*(K)}\|_2\le C\d^{-1/20}l_1^{-d/4}.
\]
In this section we show how the estimate~\eqref{equation:strongNC} implies Theorem~\ref{theorem:technical}.

In Section \ref{section:largedim} we show how to convert this information into an estimate on
$\|\rho_r(\eta)\f\|_2$ for $\f \in L^2(S^{d-1})$.
However, the scale $\d$ that we need to use depends on ``how oscillatory" $\f$ is.
Therefore, we give a Littlewood--Paley type decomposition of the space $L^2(S^{d-1})$
in Section \ref{section:LP}.
We show that the components in this decomposition are almost invariant for the operator
$S_r^{l_0}$, where $l_0$ is a suitable integer.
Then it will be enough to obtain estimates for $\|S_r^{l_0}\f\|_2$ when $\f$ belongs to one of the
components in the Littlewood--Paley decomposition.
This is done in Section \ref{section:completing} using \eqref{equation:strongNC} and the result from
section \ref{section:largedim}.

Numerous parameters will appear in the following sections.
Unfortunately, it is difficult to keep track of their interdependence,
and this feature makes the argument difficult to follow.
However, when $r$ is large (i.e. $r>M$), then the choice of these parameters
is more transparent.
Therefore, we will comment on the values of the parameters in this regime in the course of the proof.

%%%%%%%%%%%%%%%%%%%%%%%%%%%%%%%%%%%%%%%%%
\subsection{A Littlewood--Paley decomposition}
\label{section:LP}
%%%%%%%%%%%%%%%%%%%%%%%%%%%%%%%%%%%%%%%%%

We fix some positive numbers $r,L$.
Write
\be\label{equation:n0}
n_0:=\left[\frac{\log(100r\sqrt{L})}{\log 2}\right]+1.
\ee
We decompose the spaces $L^2(S^{d-1})$ as the orthogonal sum
\[
L^2(S^{d-1})=\cL_0\oplus\cL_1\oplus\ldots,
\]
where
\[
\cL_0=\cH_0\oplus\ldots\oplus\cH_{2^{n_0}}
\]
and
\[
\cL_i=\cH_{2^{i+n_0-1}+1}\oplus\ldots\oplus\cH_{2^{i+n_0}}
\]
and $\cH_i$ is the space of spherical harmonics of degree $i$.

\begin{prp}\label{proposition:LittlewoodPaley}
Let $\f\in L^2(S^{d-1})$ and write
$\f=\f_0+\f_1+\ldots$, where $\f_i\in\cL_i$ for all $i$.
Let $l_0\le L$ be an integer.
Then
\[
\|S_r^{l_0}\f\|_2^2\le\frac{1}{2}\|\f\|_2^2+3\sum_{i=0}^\infty\|S_r^{l_0}\f_i\|_2^2.
\]
\end{prp}

When $r>M$, we will set $L=CM^2\log r$ and $l_0=CM^2$, where $C$ is a constant depending
only on $d$.
Thus one may think of $n_0$ being roughly proportional to $\log r$ and $l_0$ being constant.
The reason why $L$ will be taken larger than $l_0$ is cosmetic:
In Section \ref{section:largedim} we employ two different methods to obtain norm estimates
on the spaces $\cL_i$.
With the above choice of $L$, we get matching bounds for the space $\cL_0$ with the first method
and for the space $\cL_1$ with the second method.

The rest of the section is devoted to the proof of the proposition.
We begin with a lemma on the Taylor expansion of the characters $\o_r$.

\begin{lem}\label{lemma:Taylor}
Let $k\ge0$ be an integer.
Then we can write $\omega_r(g)=\f_1+\f_2$ such that
\[
\f_1\in\cH_0\oplus\ldots\oplus\cH_{k-1}
\quad{\rm and}\quad
\|\f_2\|_\infty\le2\cdot\frac{(2\pi r|v(g)|)^k}{k!}.
\]
\end{lem}

\begin{proof}
If $k\le 4\pi r|v(g)|$, we can take $\f_1=0$ and $\f_2=\omega_r(g)$, and the claim follows since
\[
\|\f_2\|_\infty\le 1\le2\cdot\frac{(2\pi r|v(g)|)^k}{k!}
\]
by Stirling's approximation.

For the rest of the proof, we suppose that $k\ge 4\pi r|v(g)|$.
By Taylor expansion:
\[
\omega_r(g)(\xi)
=e(r\langle\xi,v(g)\rangle)
=\sum_{n=0}^{\infty}\frac{(-2\pi i r\langle\xi,v(g)\rangle)^{n}}{n!}.
\]
Write
\[
\f_1(\xi)=\sum_{n=0}^{k-1}\frac{(-2\pi i r\langle\xi,v(g)\rangle)^{n}}{n!} \quad{\rm and}\quad
\f_2=\omega_r(g)-\f_1.
\]
Then clearly $\f_1\in\cH_0\oplus\ldots\oplus\cH_{k-1}$ and
\[
|\f_2(\xi)|
\le\sum_{n=k}^\infty\frac{(2\pi r|v(g)|)^n}{n!}
\le\sum_{n=k}^\infty\frac{(2\pi r|v(g)|)^k}{k!\cdot 2^{n-k}}
\le2\cdot\frac{(2\pi r|v(g)|)^k}{k!}
\]
for all $\xi\in S^{d-1}$ which was to be proved.
\end{proof}

The next Lemma shows that $S_r\psi$ for $\psi\in\cL_i$ may have large correlations only with
functions belonging to $\cL_{i-1}\oplus\cL_i\oplus\cL_{i+1}$.
In the course of the proof we will need the following fact about the second moments
of convolutions of $\mu$:
\be\label{equation:secondmoment}
\left[\int |v(g)|^2\igap d\mu^{*(l)}(g)\right]^{1/2}\le l^{1/2.}
\ee
This can be proved easily by induction starting with \eqref{equation:moment2}.
See Lemma \ref{lemma:moments} below for details.

\begin{lem}\label{lemma:Taylor2}
Let $i\ge0$ be an integer and let
\[
\psi_1\in\cL_i \quad{\rm and}\quad
\psi_2\in\cL_{i+2}\oplus\cL_{i+3}\oplus\ldots.
\]
Then
\[
|\langle S_r^{l_0}\psi_1,S_r^{l_0}\psi_2\rangle|
\le\frac{(2\pi e r)^2\cdot2 l_0}{2^{2(n_0+i)}}\|\psi_1\|_2\|\psi_2\|_2.
\]
\end{lem}
\begin{proof}
By \eqref{equation:symmetric}, $S_r$ is selfadjoint, hence
\[
\langle S_r^{l_0}\psi_1,S_r^{l_0}\psi_2\rangle
=\langle S_r^{2l_0}\psi_1,\psi_2\rangle
=\int \langle \rho_r(g)\psi_1,\psi_2\rangle\igap d\mu^{*(2l_0)}(g).
\]
Fix $g$, and apply Lemma \ref{lemma:Taylor} to get $\omega_r(g)=\f_1+\f_2$
such that
\[
\f_1\in\cH_0\oplus\ldots\oplus\cH_{2^{n_0+i}-1}
\quad{\rm and}\quad
\|\f_2\|_\infty\le2\cdot\frac{(2\pi r|v(g)|)^{2^{n_0+i}}}{(2^{n_0+i})!}.
\]
By Stirling's approximation, we have
\[
(2^{n_0+i})!\ge 2\left(\frac{2^{n_0+i}}{e}\right)^{2^{n_0+i}},
\]
hence
\[
\|\f_2\|_\infty\le\left(\frac{2\pi e r|v(g)|}{2^{n_0+i}}\right)^{2^{n_0+i}}.
\]

We can write
\[
 \langle \rho_r(g)\psi_1,\psi_2\rangle=\langle (\f_1+\f_2)\cdot\rho_0(g)\psi_1,\psi_2\rangle.
\]
Note that
\[
\f_1\cdot\rho_0(g)\psi_1\in\cH_0\oplus\ldots\oplus \cH_{2^{i+n_0}+2^{i+n_0}-1},
\]
and hence it is orthogonal to $\psi_2$.
Then the above estimate on $\|\f_2\|_\infty$ gives
\[
| \langle \rho_r(g)\psi_1,\psi_2\rangle|
\le\left(\frac{2\pi e r|v(g)|}{2^{n_0+i}}\right)^{2^{n_0+i}}\cdot\|\psi_1\|_2\|\psi_2\|_2.
\]

On the other hand, we have the trivial estimate
\[
| \langle \rho_r(g)\psi_1,\psi_2\rangle|
\le\|\psi_1\|_2\|\psi_2\|_2,
\]
which combined with the above gives
\[
| \langle \rho_r(g)\psi_1,\psi_2\rangle|
\le\left(\frac{2\pi e r|v(g)|}{2^{n_0+i}}\right)^{2}\cdot\|\psi_1\|_2\|\psi_2\|_2.
\]

Integrating $g$ and using \eqref{equation:secondmoment},
we get
\[
|\langle S_r^{l_0}\psi_1,S_r^{l_0}\psi_2\rangle|
\le \frac{(2\pi e r)^2\cdot 2l_0}{2^{2n_0+2i}}\cdot\|\psi_1\|_2\|\psi_2\|_2,
\]
which was claimed.
\end{proof}

\begin{proof}[Proof of Proposition \ref{proposition:LittlewoodPaley}]
For simplicity, assume that $\|\f\|_2=1$.
For integers $i\ge0$ write
\[
\f_{>i+1}=\f_{i+2}+\f_{i+3}+\ldots.
\]
By simple calculation
\begin{align}
\|S_r^{l_0}\f\|_2^2
&=\sum_{i=0}^{\infty}\|S_r^{l_0}\f_i\|_2^2\nonumber\\
&\qquad{}+\sum_{i=0}^{\infty}\langle S_r^{l_0}\f_i,S_r^{l_0}\f_{i+1}\rangle
+\sum_{i=0}^{\infty}\langle S_r^{l_0}\f_{i+1},S_r^{l_0}\f_i\rangle\label{equation:type1}\\
&\qquad{}+\sum_{i=0}^{\infty}\langle S_r^{l_0}\f_i,S_r^{l_0}\f_{>i+1}\rangle
+\sum_{i=0}^{\infty}\langle S_r^{l_0}\f_{>i+1},S_r^{l_0}\f_i\rangle\label{equation:type2}.
\end{align}

To estimate \eqref{equation:type1}, we write
\[
|\langle S_r^{l_0}\f_i,S_r^{l_0}\f_{i+1}\rangle|
\le\| S_r^{l_0}\f_i\|_2\cdot\| S_r^{l_0}\f_{i+1}\|_2
\le\frac{\| S_r^{l_0}\f_i\|_2^2+\| S_r^{l_0}\f_{i+1}\|_2^2}{2}.
\]
Summing up, we get
\[
|\eqref{equation:type1}|\le2\cdot\sum_{i=0}^{\infty}\| S_r^{l_0}\f_i\|_2^2.
\]

To estimate \eqref{equation:type2}, we use Lemma \ref{lemma:Taylor2}.
We can write
\[
|\langle S_r^{l_0}\f_i,S_r^{l_0}\f_{>i+1}\rangle|
\le\frac{(2\pi e r)^2\cdot2 l_0}{2^{2(n_0+i)}}.
\]
Summing up, this yields
\[
|\eqref{equation:type2}|
\le2\cdot\sum_{i=0}^{\infty}\frac{(2\pi e r)^2\cdot2 l_0}{2^{2(n_0+i)}}
\le\frac{8}{3}\cdot\frac{(2\pi e r)^2\cdot2 l_0}{2^{2n_0}}
\le\frac{8}{3}\cdot\frac{(2\pi e r)^2\cdot2 l_0}{100^2r^2L}.
\]
For the last inequality, we used the definition of $n_0$.
Since $l_0\le L$ and
\[
\frac{8}{3}\cdot\frac{(2\pi e)^2\cdot2 }{100^2}\le\frac{1}{2},
\]
this proves the proposition.

\end{proof}

%%%%%%%%%%%%%%%%%%%%%%%%%%%%%%%%%%%%%%%%%
\subsection{Measures of large dimension}
\label{section:largedim}
%%%%%%%%%%%%%%%%%%%%%%%%%%%%%%%%%%%%%%%%%

As in Section \ref{section:LP}, we fix some positive numbers $r,L$.
Let $n_0$ and $\cL_i$ be the same as in that section.

We prove in this section that for any number $i$ and functions $\f_1,\f_2\in\cL_i$,
for most $g\in\Isom(\R^d)$, $\rho_r(g)\f_1$ and $\f_2$ are almost
orthogonal, that is $|\langle\rho_r(g)\f_1,\f_2\rangle|$ is small.
In fact, the exceptional set, where $|\langle\rho_r(g)\f_1,\f_2\rangle|$ is large will be so small
that  the ``strong non-concentra\-tion" estimate \eqref{equation:strongNC} implies that the above
inner product is small for $\eta$-typical $g\in\Isom(\R^d)$, as well.
This implies then an estimate for $\|\rho_r(\eta)\f_1\|_2^2=\langle\rho_r(\eta)^2\f_1,\f_1\rangle$.
Recall the definition of $\eta$ from the beginning of the section, in particular observe
that it is symmetric.

The purpose of this section is to prove the following proposition.

\begin{prp}\label{proposition:largedimension}
Let $R\ge 0$, and denote by $B_R$ the set of isometries $\g$ in $\Isom(\R^d)$
for some $d\ge 3$  such that $|v(\g)|<R$.
Fix some $\f_1,\f_2\in\cL_i$.
If $i=0$, we have
\[
\frac{1}{m(B_R)}\int_{B_R}|\langle\rho_r(g)\f_1,\f_2\rangle|\igap dm(g)
\le C  (rR)^{-(d-1)/2}\|\f_1\|_2\|\f_2\|_2.
\]
If $i>0$, we have
\[
\frac{1}{m(B_R)}\int_{B_R}|\langle\rho_r(g)\f_1,\f_2\rangle|\igap dm(g)
\le C 2^{-(n_0+i)(d-2)/2}\|\f_1\|_2\|\f_2\|_2.
\]
\end{prp}
Here and in what follows, we denote by $m$ the Haar measures on both $\Isom(\R^d)$ and $\SO(d)$.
On the first group we take an arbitrary normalization, on the second one we take it to be
a probability measure.

We use two different methods to establish these estimates depending on $i$.
If $i>0$, we fix the translation component of $g$ and deduce the
claim from the corresponding result for the
rotation group $\SO(d)$, i.e. Proposition \ref{proposition:BG-largedim}.

If $i=0$, this method does not give a satisfactory result, since the functions in $\cH_0$
are not oscillatory enough (equivalently, the dimension of the relevant irreducible representations
of $\SO(d)$ are not big enough).
Instead, we fix the rotation component of $g$, and look at $\langle\omega_r(v)\rho_0(\t)\f_1,\f_2\rangle$
as a function of $v$.
(Here and below, we abuse notation, and write $\omega_r(v)=\omega_r(g)$ for any $g$ with $v(g)=v$,
which is permissible as  $\omega_r(g)$ depends only on $v(g)$.)
The function $\langle\omega_r(v)\rho_0(\t)\f_1,\f_2\rangle$
is easily seen to be the Fourier transform at $rv$ of the measure supported on $S^{d-1}$ with density
$\rho_0(\t)\f_1\cdot\overline{\f_2}$.
We can estimate this via Plancherel's formula in terms of $\|\rho_0(\t)\f_1\cdot\overline{\f_2}\|_2$.
Finally we show that this $L^2$ norm can be bounded on average (for $\t$) in terms of
$\|\f_1\|_2$ and $\|\f_2\|_2$.

We begin with the case $i=0$.

\begin{lem}\label{lemma:largedimtranslation}
Let $R\ge1$, and let $\f_1,\f_2$ be continuous functions on $S^{d-1}$.
Then
\be\label{equation:largedimtranslation}
R^{-d}\int_{|v|\le R}|\langle\omega_r(v)\f_1,\f_2\rangle|\igap dv
\le C (rR)^{-(d-1)/2}\|\f_1\f_2\|_2.
\ee
\end{lem}

\begin{proof}
Denote by $\l$ the measure on $\R^d$ defined by
\[
\int f(\xi) d\l(\xi)=\int_{S^{d-1}} f(\xi)\f_1(\xi)\overline{\f_2(\xi)}\igap d\xi.
\]
Observe that
\[
\langle\omega_r(v)\f_1,\f_2\rangle
= \wh{\l}(rv).
\]

Let $F$ be a continuous compactly supported function on $\R^d$ such that
$\wh F(x)\ge 1$ for $x\le1$.
Denote by $r_0$, the smallest number such that $F$ is supported in the ball of
radius $r_0$ centered at $0$.
Set
\[
F_\rho(\xi)=\rho^d\cdot F(\rho\xi)
\]
for numbers $\rho>0$.
Then $\wh{F_\rho*\l}(x)\ge\wh\l(x)$ for $|x|\le \rho$.

We estimate $\|F_{Rr}*\l\|_2$ and then use Plancherel's formula to obtain an estimate
for the average size of its Fourier transform in the ball of radius $Rr$, which is the
left hand side of \eqref{equation:largedimtranslation}.
Denote by $\chi_\rho(\xi,\zeta)$ the function on $\R^d\times S^{d-1}$,
which is $1$ if $|\xi-\zeta|<\rho^{-1}r_0$ and $0$ otherwise.
Note the identity $\chi_\rho(\xi,\zeta)F_\rho(\zeta-\xi)=F_\rho(\zeta-\xi)$.
By the Cauchy-Schwartz inequality,
\begin{align*}
\|F_{\rho}*\l\|_{L^2(\R^d)}^2
&=\int_{\R^d}\left|\int_{S^{d-1}}F_\rho(\zeta-\xi)\f_1(\xi)\overline\f_2(\xi)\igap d\xi\right|^2\igap d\zeta\\
&=\int_{\R^d}\left|\int_{S^{d-1}}\chi_\rho(\xi,\zeta)F_\rho(\zeta-\xi)\f_1(\xi)\overline\f_2(\xi)\igap d\xi\right|^2
\igap d\zeta\\
&\le C\rho^{-d+1}\int_{\R^d}\int_{S^{d-1}}F_\rho(\zeta-\xi)^2|\f_1(\xi)\f_2(\xi)|^2\igap d\xi d\zeta\\
&\le C\rho^{-d+1}\|F_\rho\|_{L^2(\R^d)}^2\|\f_1\f_2\|_{L^2(S^{d-1})}^2\le C\rho\|\f_1\f_2\|_{L^2(S^{d-1})}^2.
\end{align*}

Using the Cauchy-Schwartz inequality and then Plancherel's formula, we get
\begin{align*}
R^{-d}\int_{|v|\le R}|\langle\omega_r(v)\f_1,\f_2\rangle|\igap dv
&=(Rr)^{-d}\int_{|x|\le Rr}\wh\l(x)\igap dx\\
&\le(Rr)^{-d}\int_{|x|\le Rr}\wh{F_{Rr}*\l}(x)\igap dx\\
&\le C(Rr)^{-d/2}\left[\int_{|x|\le Rr}(\wh{F_{Rr}*\l}(x))^2\igap dx\right]^{1/2}\\
&\le C(Rr)^{-d/2}\|F_{Rr}*\l\|_{L^2(\R^d)}\\
&\le C(Rr)^{-(d-1)/2}\|\f_1\f_2\|_{L^2(S^{d-1})}.
\end{align*}
This proves the lemma.
\end{proof}

\begin{proof}[Proof of Proposition \ref{proposition:largedimension} for $i=0$]
We can write
\begin{align}
&\frac{1}{m(B_R)}\int_{B_R}|\langle\rho_r(g)\f_1,\f_2\rangle|\igap dm(g)\nonumber\\
\label{equation:largedim0}
&\le CR^{-d}\int_{\SO(d)}\int_{|v|\le R}|\langle\omega_r(v)\rho_0(\s)\f_1,\f_2\rangle|\igap dvdm(\s).
\end{align}
Using Lemma \ref{lemma:largedimtranslation} for $\rho_0(\s)\f_1$ and $\f_2$ for each fixed
$\s \in\SO(d)$, we get
\begin{align*}
\eqref{equation:largedim0}
&\le C (rR)^{-(d-1)/2}\int_{\SO(d)}\|(\rho_0(\s)\f_1)\f_2\|_2\igap dm(\s)\\
&\le C (rR)^{-(d-1)/2}\left(\int_{\SO(d)}\|(\rho_0(\s)\f_1)\f_2\|_2^2\igap dm(\s)\right)^{1/2}\\
&=C (rR)^{-(d-1)/2}\left(\int_{\SO(d)}\int_{S^{d-1}}|\f_1(\s^{-1}\xi)\f_2(\xi)|^2\igap d\xi dm(\s)\right)^{1/2}\\
&=C (rR)^{-(d-1)/2}\|\f_1\|_2\|\f_2\|_2.
\end{align*}
\end{proof}

\begin{proof}[Proof of Proposition \ref{proposition:largedimension} for $i>0$]
We use Proposition \ref{proposition:BG-largedim} for the group $G=\SO(d)$
and for the restriction of the regular representation to the space $\cL_i$.
The irreducible components of this representation are $\cH_j$ for $j=2^{i+n_0-1}+1,\ldots,2^{i+n_0}$.
The dimension of $\cH_j$ is
\[
\binom{d+j-1}{d-1}-\binom{d+j-3}{d-1}\ge c j^{d-2}
\]
(see \cite[Sect. IV.2]{SW-Intro}).
Thus all irreducible components of $\cL_i$ are of dimension at least $c2^{(i+n_0)(d-2)}$.

Proposition \ref{proposition:BG-largedim} then gives
\begin{align*}
\int_{\SO(d)}|\langle\rho_0(\t)f_1,f_2 \rangle|\igap dm(\t)
&\le
\left[\int_{\SO(d)}|\langle\rho_0(\t)f_1,f_2 \rangle|^2\igap dm(\t)\right]^{1/2}\\
&\le C\frac{\|f_1\|_2\|f_2\|_2}{2^{(i+n_0)(d-2)/2}}.
\end{align*}
We use this inequality with $f_1=\f_1$ and $f_2=\overline{\omega_r(g)}\cdot \f_2$:
\begin{align*}
&\frac{1}{m(B_R)}\int_{B_R}|\langle\rho_r(g)\f_1,\f_2\rangle|\igap dm(g)\\
&\qquad{}=\frac{1}{m(B_R)}\int_{|v|\le R}\int_{\SO(d)}|\langle\rho_0(\t)\f_1,\overline{\omega_r(v)}\f_2\rangle|
\igap dm(\t)dv\\
&\qquad{}\le C\frac{\|\f_1\|_2\|\f_2\|_2}{2^{(i+n_0)(d-2)/2}}.
\end{align*}
\end{proof}

%%%%%%%%%%%%%%%%%%%%%%%%%%%%%%%%%%%%%%%%%%
\subsection{Completing the proof}
\label{section:completing}
%%%%%%%%%%%%%%%%%%%%%%%%%%%%%%%%%%%%%%%%%%

As in the previous sections, we fix some numbers $r,L$ and let $n_0$ and $\cL_i$ be as
defined in Section \ref{section:LP}.
In addition, we fix some number $i$, and a function $\f_i\in\cL_i$.
Recall the definition $\eta=P_{\d,l_1}*\mu^{*(l_1)}*P_{\d,l_1}$ from the beginning of the section.

Our aim in this section is to prove the following proposition.
\begin{prp}\label{proposition:l0}
There is a number $C$ depending only on $d\ge3$, and an integer
\[
l_0\le C\max(M^2,r^{-2})
\]
such that
\[
\|S_r^{l_0}\f_i\|_2\le\frac{1}{10}\|\f_i\|_2.
\]
\end{prp}

We combine this with Proposition \ref{proposition:LittlewoodPaley}
and get
\[
\|S_r^{l_0}\f\|_2\le\frac{3}{4}\|\f\|_2
\]
for all $\f\in L^2(S^{d-1})$.
The bound on $l_0$ now clearly implies Theorem \ref{theorem:technical}.
Recall that $S_r$ is selfadjoint.

Therefore, it remains to prove Proposition \ref{proposition:l0}.
To simplify the notation, we omit the subscript and write $\f$ instead of $\f_i$.
Moreover, we assume that $\|\f\|_2=1$.

We set $\d=2^{-2(n_0+i)}$.
(Recall the definition of $n_0$ from Section \ref{section:LP}.)
We already mentioned that when $r$ is large, we will put $L=CM^2\log r$ and this implies that
$n_0$ is proportional to $\log r$.
In the same regime, we will take $l_1=CM^2(n_0+i)$;
$l_1$ is the parameter that appears in the definition of $\eta$.
However, in any case, we will choose $l_1$ in such a way that the condition
\be\label{equation:condition-l1}
l_1\ge CM^2\log \d^{-1}
\ee
of Proposition \ref{proposition:flattening2}
is satisfied.
In addition, we will stipulate two more conditions on the parameters later.
We will check at the end of the proof that the conditions hold with a suitable choice of the parameters.

As we noted at the beginning of Section \ref{section:largedim}, Proposition \ref{proposition:largedimension} allows us
to convert \eqref{equation:strongNC}
into an upper bound on $\|\rho_r(\eta^{*(2K)})\f\|_2$ with $K$ as in \eqref{equation:strongNC} and hence
into a bound on $\|\rho_r(\eta)\f\|_2$.
This is done in the next lemma.

\begin{lem}
Suppose that  \eqref{equation:condition-l1} holds and
\be\label{equation:condition-L}
rl_1^{1/2}2^{n_0/12}\ge 2^{n_0}.
\ee
Then
\be\label{equation:eta-estimate}
\|\rho_r(\eta)\f\|_2\le C 2^{-(n_0+i)/(40K)}.
\ee
\end{lem}

\begin{proof}
We set $R=l_1^{1/2}2^{(n_0+i)/12}$ and estimate $\eta^{*(2K)}(g:v(g)>R)$.
By Markov's inequality and the second moment bound \eqref{equation:secondmoment},
\[
\mu^{*(l_1)}(\{g\in\Isom(\R^d):v(g)>R/(2K)-1\})\le C 2^{-(n_0+i)/6}.
\]
This in turn implies
\be\label{equation:outsidetheball}
\eta^{*(2K)}(\{g\in\Isom(\R^d):v(g)>R\})\le C 2^{-(n_0+i)/6}.
\ee
By the Cauchy-Schwartz inequality and \eqref{equation:strongNC}, we then have
\be\label{equation:eta-supnorm}
\|\eta^{*(2K)}\|_\infty\le C\d^{-1/10}l_1^{-d/2}=C2^{(n_0+i)/5}l_1^{-d/2}.
\ee
We take $\f_1=\f$ and $\f_2=\rho_r(\eta^{*(2K)})\f/\|\rho_r(\eta^{*(2K)})\f\|_2$ and 
use~\eqref{equation:eta-supnorm} and~\eqref{equation:outsidetheball}:
\begin{align*}
&\|\rho_r(\eta^{*(2K)})\f\|_2=|\langle\rho_r(\eta^{*(2K)})\f_1,\f_2\rangle|\\
&\qquad{}\le\int|\langle\rho_r(g)\f_1,\f_2\rangle|d\eta^{*(2K)}(g)\\ 
&\qquad{}\le C\int_{B_R}|\langle\rho_r(g)\f_1,\f_2\rangle|\cdot2^{(n_0+i)/5}l_1^{-d/2}\igap dm(g)
+C 2^{-(n_0+i)/6}.
\end{align*}

If $i=0$ then
$rR\ge 2^{n_0}$ by assumption \eqref{equation:condition-L}.
Note also that $m(B_R)\le CR^{d}=Cl_1^{d/2}2^{(n_0+i)d/12}$.
Then Proposition \ref{proposition:largedimension} yields:
\begin{align*}
\|\rho_r(\eta^{*(2K)})\f\|_2
&\le Cm(B_R) 2^{-(n_0+i)(d-2)/2}\cdot2^{(n_0+i)/5}l_1^{-d/2}+C 2^{-(n_0+i)/6}\\
&\le C 2^{(n_0+i)d/12-(n_0+i)(d-2)/2+(n_0+i)/5}+C 2^{-(n_0+i)/6}\\
&\le C 2^{-(n_0+i)/20}.
\end{align*}
This in turn gives the claim.
\end{proof}

Recall that $\eta= P_{\d,l_1}*\mu^{*(l_1)}* P_{\d,l_1}$,
hence  $\rho_r(\eta)=\rho_r(P_{\d,l_1})S_r^{l_1}\rho_r(P_{\d,l_1})$.
We will use the next lemma to show that $\rho_r(P_{\d,l_1})\f$ is very close to $\f$, and turn
\eqref{equation:eta-estimate} into an estimate on $S_r^{l_1}\f$.

\begin{lem}\label{lemma:delta}
Let $g\in B_{\d,l_1}$.
Then
\[
\|\rho_r(g)\f-\f\|_2\le C\d rl_1^{1/2}+C\d\cdot 2^{n_0+i}.
\]
\end{lem}
\begin{proof}
We can write
\[
\|\rho_r(g)\f-\f\|_2\le\|\omega_r(g)\rho_0(g)\f-\rho_0(g)\f\|_2+\|\rho_0(g)\f-\f\|_2.
\]
The first term is bounded by
\[
\|\omega_r(g)\rho_0(g)\f-\rho_0(g)\f\|_2\le\|\omega_r(g)-1\|_\infty\le Cr|v(g)|\le C\d rl_1^{1/2}.
\]
Hence, it is left to estimate the second term.

Fix $g\in B_{\d,l_1}$ and let $T$ be a maximal torus in $\SO(d)$
which contains $\t(g)$, and denote by $\ft$ its Lie algebra.
We denote by $\Lambda$ the kernel of the exponential map on $\ft$, which is a lattice in $\ft$.
We decompose $\cL_i$ as the sum of weight spaces for $T$, that is we write
$\cL=V_{w_1}\oplus\ldots\oplus V_{w_n}$, where $w_1,\ldots,w_n\in \Lambda^*$ are the weights
of the representation $\rho_0$ on $\cL_i$ and
\[
\rho_0(\s)\psi=e^{2\pi i\langle\log \s,w_k\rangle}\psi
\]
for $\s\in T$ and $\psi\in V_k$.
The highest weight is of length $C 2^{n_0+i}$ in $\cL_i$,
since these functions are restrictions of polynomials of degree at most $2^{n_0+i}$.
Hence $|w_k|\le C2^{n_0+i}$.

We decompose $\f=\psi_1+\ldots+\psi_n$, where $\psi_k\in V_{w_k}$.
Then
\begin{align*}
\|\rho_0(g)\f-\f\|_2^2&=\sum_{k=1}^n|1-e^{2\pi i\langle\log \t(g),w_k\rangle}|^2\|\psi_k\|_2^2\\
&\le C(2^{n_0+i}\d)^2.
\end{align*}
This proves the lemma.
\end{proof}

\begin{lem}
Suppose that \eqref{equation:condition-l1} and \eqref{equation:condition-L} hold and
\be\label{equation:conditionL2}
rl_1^{1/2}\le 2^{n_0+i}.
\ee
Then
\be\label{equation:Srl1}
\|S_r^{l_1}\f\|_2\le  C 2^{-(n_0+i)/(80K)}.
\ee
\end{lem}
\begin{proof}
Recall that $\d=2^{-2(n_0+i)}$.
Then Lemma \ref{lemma:delta} together with \eqref{equation:conditionL2} implies that
\[
\|\rho_r(g)\f-\f\|_2\le C 2^{-n_0-i}
\]
for $g\in B_{r,l_1}$, hence
\begin{equation}\label{equation:rhor Pdelta}
\|\rho_r(P_{\d,l_1})\f-\f\|_2\le  C 2^{-n_0-i}.
\end{equation}

By the trinagle inequality,
\begin{align*}
|\langle S_r^{l_1}\f,\f\rangle|
\le& |\langle S_r^{l_1}\rho_r(P_{\d,l_1})\f,\rho_r(P_{\d,l_1})\f\rangle|\\
&+|\langle S_r^{l_1}\rho_r(P_{\d,l_1})\f,\f-\rho_r(P_{\d,l_1})\f\rangle|
+|\langle S_r^{l_1}(\f-\rho_r(P_{\d,l_1})\f),\f\rangle|\\
\le& |\langle\rho_r(\eta)\f,\f\rangle|+2\|\f-\rho_r(P_{\d,l_1})\f\|_2\\
\le& C 2^{-(n_0+i)/(40K)}.
\end{align*}
For the last inequality, we used  \eqref{equation:eta-estimate} and \eqref{equation:rhor Pdelta}.

We observe that $S_r$ is a positive self-adjoint operator of norm at most $1$ owing to the assumption
\eqref{equation:symmetric}.
Thus
\[
\|S_r^{l_1}\f\|_2^2=\langle S_r^{l_1}\f,S_r^{l_1}\f\rangle=\langle S_r^{2l_1}\f,\f\rangle
\le\langle S_r^{l_1}\f,\f\rangle,
\]
which proves the claim.
\end{proof}

We show how to set the parameters $l_1$ and $L$ in such a way that the conditions
imposed on these parameters, namely \eqref{equation:condition-l1},
\eqref{equation:condition-L} and~\eqref{equation:conditionL2} hold.
There are two cases depending on the size of $r$.
\begin{lem}\label{lemma:defs}
Let $A$ be a number and set
\[
L:=A\cdot\left\{
\begin{array}{ll}
M^2(\log(M)+\log(r)+1)&\text{if $M^{-1}\le r$},\\
r^{-2}&\text{if $r< M^{-1}$}.
\end{array}
\right.
\]
and
\[
l_1:=10^4(L+A^{1/2}M^2i).
\]
If $A$ is sufficiently large depending only on $d$, then the conditions
\eqref{equation:condition-l1}, \eqref{equation:condition-L} and \eqref{equation:conditionL2} hold.
\end{lem}

\begin{proof}
To establish \eqref{equation:condition-L}, we write
\[
rl_1^{1/2}2^{n_0/12}\ge100L^{1/2}r2^{n_0/12}.
\]
Recall the definition of $n_0$ in \eqref{equation:n0}, in particular that
\be\label{equation:n}
100rL^{1/2}\le 2^{n_0}\le 200 rL^{1/2}.
\ee
We see that \eqref{equation:condition-L} holds as long as  $n_0\ge 12$
that we can ensure by choosing $A$ large enough.

Inspecting the definition of $L$, we see that $L\ge AM^2\ge A^{1/2}M^2$ for all $r$.
Hence $l_1\le 10^4 L(i+1)\le 10^4 L \cdot 2^{2i}$ which yields
\[
rl_1^{1/2}\le 100r L^{1/2}\cdot 2^i\le 2^{n_0+i},
\]
which is precisely \eqref{equation:conditionL2}.

It remains to verify \eqref{equation:condition-l1}.
Recall that $\d=2^{-2(n_0+i)}$ by definition.
Hence \eqref{equation:condition-l1} would follow from
\[
l_1=10^4(L+A^{1/2}M^2i)\ge CM^2(n_0+i).
\]
We see that this condition holds for all
$i$, if it holds for $i=0$ and $A$ is sufficiently large.

We verify the condition for $i=0$ and consider the two ranges for $r$
separately.
First, we consider the case $M^{-1}\le r $.
Then
\begin{align*}
n_0&\le C+\log r + \frac{1}{2}\log L\\
&\le C+\log r +\log M + \log A +\log(\log  M+\log r +1)\\
&\le CA^{1/2}(\log  M+\log r +1),
\end{align*}
which implies  \eqref{equation:condition-l1} if $A$ is sufficiently large.

Second, let $r<M^{-1}$.
Then
\[
n_0\le C+\log r + \frac{1}{2}\log L
\le C +\log r +\log r^{-1}+\log A
\le CA^{1/2}.
\]
Thus
\[
M^2n_0\le CA^{1/2}r^{-2},
\]
which implies  \eqref{equation:condition-l1} if $A$ is sufficiently large.
\end{proof}

\begin{proof}[Proof of Proposition \ref{proposition:l0}]
We take the definitions of $L$ and $l_1$ from Lemma \ref{lemma:defs}, so in particular
\eqref{equation:Srl1} holds.
We observe that if $A$ is sufficiently large (depending on $d$), then $n_0$ will be bigger
than any number $C$ which depends only on $d$, (see \eqref{equation:n}).
Then \eqref{equation:Srl1} implies
\be\label{equation:Srl1-2}
\|S_r^{l_1}\f\|_2\le 2^{-(n_0+i)/(100K)}.
\ee

Now we fix a number $B$ that will be chosen sufficiently large and  set
\[
l_0=\left\{
\begin{array}{ll}
ABM^2&\text{if $r>e^B/M$,}\\
L&\text{otherwise.}
\end{array}
\right.
\]
We claim that
\be\label{equation:Srl0}
\|S_r^{l_0}\f\|_2\le\frac{1}{10},
\ee
if $A$ is sufficiently large depending on $d$ and $B$ is sufficiently large depending on $d$ and $A$.

{}From \eqref{equation:Srl1-2}, we have
\[
\|S_r^{l_0}\f\|_2\le 2^{-\frac{(n_0+i)l_0}{100K l_1}}.
\]
Hence to prove the claim, we only need to show that
\[
\frac{(n_0+i)l_0}{l_1}
\]
can be arbitrarily large with a suitable choice of $A$ and $B$.

By \eqref{equation:n} and the definition of $l_1$, we have
\begin{align*}
\frac{n_0+i}{l_1}&\ge\frac{\log r+(1/2)\log L+i}{10^4(L+A^{1/2}M^2i)}\\
&\ge\frac{1}{2\cdot10^4}\min\left\{\frac{2\log r+\log L}{L},\frac{1}{A^{1/2}M^2}\right\}.
\end{align*}
We observe that $L\ge AM^2$ for all $r$, hence $l_0\ge AM^2$ as long as $B\ge 1$.
Hence
\[
\frac{1}{A^{1/2}M^2}\cdot l_0\ge A^{1/2}
\]
is as large as we wish.
So it is left to show that
\[
\frac{2\log r+\log L}{L}\cdot l_0
\]
can be arbitrarily large, as well.

If $l_0=L$ this follows from the inequality $2\log r+\log L\ge \log A$.
In the opposite case $r\ge1/M$, hence
$L=AM^2(\log(M)+\log(r)+1)$, and $\log L\ge 2\log M +2$.
Then
\[
\frac{2\log r+\log L}{L}\cdot l_0
\ge\frac{2\log r+2\log M+2}{AM^2(\log(M)+\log(r)+1)}\cdot ABM^2
\ge2B,
\]
which can be arbitrarily large.
\end{proof}

%%%%%%%%%%%%%%%%%%%%%%%%%%%%%%%%%%%%%%%%%%%%
\section{A more general form}
\label{section:general}
%%%%%%%%%%%%%%%%%%%%%%%%%%%%%%%%%%%%%%%%%%%%

In this section we give a more general form of Theorem \ref{theorem:technical}, which
does not require the assumptions \eqref{equation:symmetric}--\eqref{equation:moment3}.
This is the form that will be used in the next two sections, and it is a rather straightforward
consequence of Theorem \ref{theorem:technical};
its proof consists of a series of simple observations, which reduces the general setting to
\eqref{equation:symmetric}--\eqref{equation:moment3}.

\begin{cor}\label{corollary:general}
Let $\mu$ be a probability measure on $\Isom(\R^d)$ for some $d\ge 3$,
and let $v_1,v_2\in\R^d$ be points 
such that
\[
N^2:=\int |g(v_1)-v_2|^2 \igap d\mu(g)
\]
is minimal.
Suppose that $N>0$ and
\[
\int |g(v_1)-v_2|^3 \igap d\mu(g)\le MN^3
\]
for some number $M<\infty$.
Then there is a number $c$ depending only on $d$ such that
\[
\|S_r\|\le 1- c\min\left\{(Nr)^{2},\frac{1-\|T\|_{L^2_0(\SO(d))}}{M^2}\right\}.
\]
\end{cor}
Here $T$ is the averaging operator on $L^2_0(\SO(d))$ given by~\eqref{equation:defT}.

In the next lemma we record some simple facts about the growth of the $m$th moments
of the translation part of a product of $l$ independent isometries.
This follows easily from some general inequalities of Burkholder on martingales.
This lemma will be used only for $m=2,3$.

\begin{lem}\label{lemma:moments}
Let $Z_1,Z_2,\ldots\in\Isom(\R^d)$ be a sequence of independent
(not necessarily identically distributed) random isometries.
Suppose that $\E(Z_i(0))=0$ for all $i\ge1$.
Then
\be\label{equation:second}
\E(|Z_1\cdots Z_l(0)|^2)=\sum_{i=1}^l\E[|Z_i(0)|^2]
\ee
and
\be\label{equation:mth}
\E(|Z_1\cdots Z_l(0)|^m)^{2/m}
\le C_m\sum_{i=1}^l\left(\E[|Z_i(0)|^m]\right)^{2/m},
\ee
for any $m\ge 2$ and integer $l\ge1$, where $C_m$ is a number depending only on $m$ and $d$.
\end{lem}

The lemma is just a more explicit form of \cite{Var-Euclidean}*{Lemma 35}, but we give
the proof for completeness.

\begin{proof}
For a vector $v\in\R^d$, we write $v[1],\ldots,v[d]$ for its $d$ coordinates.
We consider the sequence of random vectors
\[
V_l=Z_1\cdots Z_l(0)=v(Z_1)+\t(Z_1)v(Z_2)
+\ldots+\t(Z_1)\cdots\t(Z_{l-1})v(Z_l).
\]
Since $Z_l$ is independent of $Z_1,\ldots,Z_{l-1}$ and $\E[v(Z_l)]=0$,
it follows that
\[
\E[\t(Z_1)\cdots\t(Z_{l-1})v(Z_l)|Z_1,\ldots,Z_{l-1}]=0.
\]
Thus $\E[V_l|V_{l-1}]=V_{l-1}$, and the coordinate functions $V_l[k]$ form martingales
for all $1\le k\le d$.

For fixed $k$, the random variables $V_1[k],V_2[k]-V_1[k],\ldots,V_l[k]-V_{l-1}[k]$ are orthogonal in the
$L^2$ space of the underlying probability space.
This implies the first claim \eqref{equation:second}.

The second claim depends on Burkholder's inequality \cite{Bur-martingale}.
By \cite[Theorem 3.2]{Bur-martingale}, we have
\[
\E[|V_l[k]|^m]\le C_m\E\left[\left(\sum_{i=1}^l|V_i[k]-V_{i-1}[k]|^2\right)^{m/2}\right],
\]
where $C_m$ is a constant depending only on m.
By Minkowski's inequality, we have
\be\label{equation:coordinateineq}
\E[|V_l[k]|^m]^{2/m}\le C_m\sum_{i=1}^l\E[|V_i[k]-V_{i-1}[k]|^m]^{2/m}.
\ee

For any vector $v\in\R^d$ and numbers $0<s\le t$, we have
\[
\left(\sum_{k=1}^d |v[k]|^t\right)^{1/t}\le\left(\sum_{k=1}^d |v[k]|^s\right)^{1/s}
\le d^{1/s-1/t}\left(\sum_{k=1}^d |v[k]|^t\right)^{1/t}.
\]
We sum the $m/2$th power of \eqref{equation:coordinateineq} for $k=1,\ldots,d$ and use the above
inequality several times:
\begin{align*}
&d^{1-m/2}\E[|V_l|^m]\le\sum_{k=1}^d\E[|V_l[k]|^m]\\
&\qquad{}\le C_m^{m/2}\sum_{k=1}^d\left(\sum_{i=1}^l\E[|V_i[k]-V_{i-1}[k]|^m]^{2/m}\right)^{m/2}\\
&\qquad{}\le C_m^{m/2}\left(\sum_{i=1}^l\sum_{k=1}^d\E[|V_i[k]-V_{i-1}[k]|^m]^{2/m}\right)^{m/2}\\
&\qquad{}\le  C_m^{m/2}d^{m/2-1}\left(\sum_{i=1}^l
\left(\sum_{k=1}^d\E[|V_i[k]-V_{i-1}[k]|^m]\right)^{2/m}\right)^{m/2}\\
&\qquad{}\le  C_m^{m/2}d^{m/2-1}\left(\sum_{i=1}^l\E[|V_i-V_{i-1}|^m]^{2/m}\right)^{m/2}
\end{align*}
Upon taking $2/m$th power of both end, we get \eqref{equation:mth}.
\end{proof}

\begin{proof}[Proof of Corollary \ref{corollary:general}]
We write $\mu_1:=\check \mu*\mu$.
The Corollary is vacuous if $T$ has no spectral gap;
hence there are no two units vectors $f_1, f_2$ such that $\mu$-almost surely $\theta(g)f_1=f_2$.
It follows that $\norm{\int \theta(g) \igap d \mu_1(g)}<1$
hence there is a unique point $x_0 \in \R^d$
such that
$\int g(x_0)\igap d \mu_1(g)=x_0$; cf. \cite{Var-Euclidean}*{Lemma 20} for more details. 
Without loss of generality, we can
assume that $x_0=0$ by conjugating $\mu$ by an isometry that maps $x_0$ to $0$.
Note this conjugation
does not change the norm of $\rho_r(\mu)$ nor
the values of $N$ and $M$ in the assumptions.

We write
\[
N_1^2:=\int|v(g)|^2\igap d\mu_1(g)=\int|g_1(0)-g_2(0)|^2\igap d\mu(g_1)d\mu(g_2)\ge N^2.
\]
In addition,
\begin{align*}
\int|g(v_1)|^3\igap d\mu_1(g)&=\int|g_1(v_1)-g_2(v_1)|^3\igap d\mu(g_1)d\mu(g_2)\\
&\le8\int|g_1(v_1)-v_2|^3\igap d\mu(g_1)\le 8MN^3.
\end{align*}
Moreover
\begin{align*}
\int|v(g)|^3\igap d\mu_1^{*(2)}(g)&=\int|g_1(0)-g_2(0)|^3\igap d\mu_1(g_1)d\mu_1(g_2)\\
&\le 8\int|g_1(0)-v_1|^3\igap d\mu_1(g_1)\\
&= 8\int|g_1^{-1}(v_1)-0|^3\igap d\mu_1(g_1)\le 64MN^3.
\end{align*}

We put $l=\lceil (1-\|T\|_{L_0^2(\SO(d))})^{-1}\rceil+1$ and
$\mu_2:=\mu_1^{*(2l)}$.
We apply Lemma \ref{lemma:moments} and get
\[
\int |v(g)|^2\igap d\mu_2(g)=2l N_1^2
\]
and
\[
\int |v(g)|^3\igap d\mu_2(g)\le Cl^{3/2} MN^3\le CM(l^{1/2}N_1)^3.
\]

We consider the map $\Phi:\Isom(\R^d)\to\Isom(\R^d)$, which does not change the
rotation part of the isometry and dilates the translation part in accordance with
$v(\Phi(g))=v(g)/(2l)^{1/2}N_1$.
We define $\mu_3$ to be the pushforward of $\mu_2$ via $\Phi$.
Then
\[
\int |v(g)|^2\igap d\mu_3(g)=\int |v(\Phi(g))|^2\igap d\mu_2(g)=1
\]
and
\[
\int |v(g)|^3\igap d\mu_3(g)= \int|v(\Phi(g))|^2\igap d\mu_2(g)\le CM.
\]

We see that assumptions \eqref{equation:symmetric} and \eqref{equation:centered}--\eqref{equation:moment3}
hold for $\mu_3$ in place of $\mu$.
Moreover,
\[
\|\cR_0(\t(\mu_3))\|=\|\cR_0(\t(\mu_2))\|=\|\cR_0(\t(\mu_1))\|^{2l}=\|\cR_0(\t(\mu))\|^{4l}\le1/2
\]
by the choice of $l$.
(Recall that $\cR_0$ denotes the regular representation of $\SO(d)$ restricted to $L_0^2(\SO(d))$.)
Thus \eqref{equation:gap} also holds for $\mu_3$.

Then we can apply Theorem \ref{theorem:technical}, and get
\[
\|\rho_r(\mu_3)\|\le 1-c\min\{r^2, M^{-2}\}.
\]
Thus
\[
\|\rho_r(\mu_2)\|=\|\rho_{(2l)^{1/2}N_1r}(\mu_3)\|\le1-c\min\{2l(Nr)^2, M^{-2}\}.
\]
Finally, we note that
\[
\|S_r\|=\|\rho_r(\mu)\|=\|\rho_r(\mu_1)\|^{1/2}=\|\rho_r(\mu_2)\|^{1/4l}.
\]
This proves the corollary by the choice of $l$.
\end{proof}

%%%%%%%%%%%%%%%%%%%%%%%%%%%%%%%%%%%%%%%%%%%%
\section{Random walks}
\label{section:rwproof}
%%%%%%%%%%%%%%%%%%%%%%%%%%%%%%%%%%%%%%%%%%%%

The purpose of this section is to prove Theorem \ref{theorem:RW}.
This is relatively easy using Theorem \ref{theorem:technical} and the results of the paper \cite{Var-Euclidean}.

We denote by $\mu$ the law of $X_1$.
Then  the law of $Y_l$ is $\nu_l:=\mu^{*(l)}.\d_{x_0}$.
The proof is based on Plancherel's formula:
\[
\E[f(Y_l)]=\int_{\R^d}\wh f(\xi)\wh\nu_l(\xi)\igap d\xi.
\]
In this section, the constants $c,C$ may also depend on $\mu$ in addition to $d$.

As we have already noted in Section \ref{section:notation}, we can
understand $\wh\nu_l$ using the operators $S_r$.
In particular, we have the identity
\[
\Res_r(\wh\nu_l)=S_r(\Res_r\wh\nu_{l-1}).
\]
Using this and Theorem \ref{theorem:technical} iteratively, we can estimate $\wh\nu_l$ from above.
For high frequencies, this is sufficient and will yield the second error term in the theorem.

For low frequencies, we need more precise information, and we obtain this from
\cite{Var-Euclidean}*{Proposition 19}.
We check that the conditions of that proposition hold.
First we note that by a suitable choice of the origin
(see \cite{Var-Euclidean}*{Lemma 20}), we can assume that condition \eqref{equation:centered}
holds, which is denoted by (C) in the paper \cite{Var-Euclidean}.
In that paper, $K$ denotes the closure of the group generated by $\supp(\t(X_1))$, so $K=\SO(d)$
in our setting.
Otherwise, the operator $T$ defined in \eqref{equation:defT} would have many eigenfunctions with eigenvalue 1
contradicting to the spectral gap assumption.
Then $K$ is a semisimple group and its action fixes only the origin in $\R^d$, hence the condition (SSR)
is satisfied with the notation of \cite{Var-Euclidean}.
The condition (E) of that paper is also verified easily.

Write
\[
\psi_r(\xi):=e(r\langle\xi,x_0\rangle)=\Res_r(\wh\d_{x_0}).
\]
Then \cite{Var-Euclidean}*{Proposition 19} shows that there is a number $\s>0$ depending on $\mu$
such that
\be\label{equation:lowfreq}
\|S_r^l\psi_r-e^{-r^2l\s}\|_2<C(r^{\a-2}+|x_0|^2r^2)\cdot(e^{-clr^2}+r^{10d}).
\ee

There is a centrally symmetric Gaussian random variable $Z$ such that
\[
\E(e(\langle\xi,Z\rangle))=e^{-|\xi|^2\s}.
\]
Then by Plancherel's formula,
\[
\E[f(Y_l)]-\E[f(\sqrt l Z)]
=\int_{\R^d}\wh f(\xi)(\wh\nu_l(\xi)-e^{-|\xi|^2l\s})\igap d\xi.
\]
We estimate this integral first on the region $|\xi|<l^{-1/3}$.
We use $|\wh f(\xi)|\le\|f\|_1$ and write the integral in polar coordinates:
\begin{align*}
&\int_{|\xi|<l^{-1/3}}\wh f(\xi)(\wh\nu_l(\xi)-e^{-|\xi|^2l\s})\igap d\xi\\
&\qquad{}\le C\|f\|_1\cdot\int_{0}^{l^{-1/3}}r^{d-1}\|\Res_r(\wh\nu_l)-e^{-r^2l\s}\|_1\igap dr\\
&\qquad{}\le C\|f\|_1\cdot\int_{0}^{l^{-1/3}}r^{d-1}\|S_r^l\psi_r-e^{-r^2l\s}\|_2\igap dr\\
&\qquad{}\le C\|f\|_1\cdot\int_{0}^{l^{-1/3}}(r^{d+\a-3}+|x_0|^2r^{d+1})\cdot(e^{-clr^2}+r^{10d})\igap dr\\
&\qquad{}\le C\|f\|_1\cdot (l^{-\frac{d+\a-2}{2}}+|x_0|^2l^{-\frac{d+2}{2}}).
\end{align*}
We recognize the last expression as the first error term in Theorem \ref{theorem:RW}.

We note that
\[
\int_{|\xi|>l^{-1/3}}|\wh f(\xi)e^{-|\xi|^2l\s}|\igap d\xi
\le C\|f\|_1\cdot e^{-\s l^{1/3}}
\]
is bounded by the first error term, hence
it remains to show that
\be\label{equation:toshow}
\int_{|\xi|>l^{-1/3}}|\wh f(\xi)\wh\nu_l(\xi)|\igap d\xi
\le Ce^{-cl}\|f\|_{W^{2,(d+1)/2}}+C l^{-\frac{d+\a-2}{2}}\|f\|_1.
\ee

To this end, we estimate $\|S_r^l\|$ using Theorem \ref{theorem:technical}.

\begin{lem}\label{lemma:RW}
Suppose that the assumptions of Theorem \ref{theorem:RW} hold.
Then there is an integer $l_1$ depending only on the law of $X_1$ such that $\|S_r^{l_1}\|<1-\min\{r^2,1/2\}$.
\end{lem}

\begin{proof}[Proof of Lemma \ref{lemma:RW}]
First we choose an integer $l_0$ such that $\|T^{l_0}\|_{L^{2}_0}<1/4$.
This is possible by the assumption that $T$ has spectral gap.
Then we fix a number $R$ and denote by $B_R$ the set of isometries $g$ with $|v(g)|\le R$.
We choose $R$ in such a way that $\mu^{*(l_0)}(B_R)>3/4$, which holds if $R$ is sufficiently large depending on $\mu$
and $l_0$.
We write $\mu_0$ for the restriction of $\mu^{*(l_0)}$ to $B_R$ renormalized to be a probability measure,
that is
\[
\mu_0(A):=\frac{\mu^{*(l_0)}(A\cap B_R)}{\mu^{*(l_0)}(B_R)}
\]
for every Borel set $A\subset\Isom(\R^d)$.
We assume that there is no point $x\in\R^d$ which is fixed by $\mu_0$-almost all $g$.
Again, this holds if $R$ is sufficiently large.

Recall that $\cR_0$ denotes the regular representation of $\SO(d)$ restricted to $L^2_0(\SO(d))$.
Then $\cR_0(\t(\mu))^{l_0}=(3/4)\cR_0(\t(\mu_0))+(1/4) X$, where $X$ is an operator of norm at most 1.
Hence $\|\cR_0(\t(\mu_0))\|\le 2/3$.
Thus we can apply Corollary \ref{corollary:general} and get
\[
\|\rho_r(\mu_0)\|\le1-c\min\{r^2,1/2\}.
\]
with a constant $c$ depending only on $\mu$.
(The condition $N>0$ holds in Corollary \ref{corollary:general} thanks to our
assumption that there is no point $x\in\R^d$ which is fixed by $\mu_0$-almost all $g$.)

Since $S_r^{l_0}=(3/4)\rho_r(\mu_0)+(1/4)Y$, where $Y$ is an operator of norm at most 1, we have
\[
\|S_r^{l_0}\|\le1-c\min\{r^2,1/2\}.
\]
This proves the lemma, if $l_1$ is a suitably large integer.
\end{proof}

We return to \eqref{equation:toshow} and write the left side in polar coordinates and use the Cauchy-Schwartz
inequality.
\begin{align*}
\int_{a<|\xi|<b}|\wh f(\xi)\wh\nu_l(\xi)|\igap d\xi
&\le C\int_{a}^b r^{d-1}\int_{S^{d-1}}|\Res_r(\wh f)(\xi)S_r^l\psi_r(\xi)|\igap d\xi dr\\
&\le C\int_{a}^b r^{d-1}\|\Res_r(\wh f)\|_2\|S_r^l\psi_r\|_2\igap dr.
\end{align*}
We plug in Lemma \ref{lemma:RW} and suppose that $l\ge2l_1$:
\[
\int_{a<|\xi|<B}|\wh f(\xi)\wh\nu_l(\xi)|\igap d\xi
\le C\int_{a}^b r^{d-1}e^{-(l/2l_1)\min\{r^2,1/2\}}\|\Res_r(\wh f)\|_2\igap dr.
\]

For $l^{-1/3}<r<1$, we have $r^{d-1}e^{-(l/2l_1)\min\{r^2,1/2\}}\le e^{-cl^{1/3}}$, hence
\[
\int_{l^{-1/3}<|\xi|<1}|\wh f(\xi)\wh\nu_l(\xi)|\igap d\xi
\le Ce^{-cl^{1/3}}\|f\|_1,
\]
which is dominated by the first error term.
So it is left to consider the domain $|\xi|>1$.

We use the Cauchy-Schwartz inequality and transform the integral back to Cartesian coordinates.
\begin{align*}
\int_{1<|\xi|}|\wh f(\xi)\wh\nu_l(\xi)|\igap d\xi
&\le Ce^{-cl}\left[\int_{1}^\infty r^{2d}\|\Res_r(\wh f)\|_2^2\igap dr\right]^{1/2}
\cdot\left[\int_{1}^{\infty}r^{-2}\igap dr\right]^{1/2}\\
&\le Ce^{-cl}\left[\int_{|\xi|>1} |\xi|^{d+1}|\wh f(\xi)|^2\igap d\xi\right]^{1/2},
\end{align*}
which is dominated by the second error term.
This finishes the proof of Theorem \ref{theorem:RW}.

\begin{proof}[Proof of Corollary \ref{corollary:RW}]
Let $f_{r,z}$ be a nonnegative smooth function supported on $B(r(1+1/l),z)$
such that $f_{r,z}(x)=1$ for $x\in B(r,z)$ and $\|f_{r,z}\|_1\le C r^d$
and $\|f_{r,z}\|_{W^{2,(d+1)/2}}\le C l^{d/2} r^{-1/2}$, where $C$ is a number that depends only on
$d$.
Then
\[
\E(f_{r(1+1/l)^{-1},z}(Y_l))\le\P(Y_l\in B(r,z))\le\E(f_{r,z}(Y_l)).
\]
We note that
\begin{align*}
&\P(\sqrt l Z+y_0\in B(r(1+1/l)^{-1},z)))\le\E(f_{r(1+1/l)^{-1},z}(\sqrt l Z+y_0))\\
&\qquad{}\le\E(f_{r,z}(\sqrt l Z+y_0))
\le\P(\sqrt l Z+y_0\in B(r(1+1/l),z))
\end{align*}
and
\begin{align*}
\P(\sqrt l Z+y_0\in B(r(1+1/l)^{\pm1},z)&=r^dl^{-d/2}\frac{e^{- |y_0-z|^2/2l\s^2}}{\sqrt{(2\pi)^k\s^{2d}}}\\
&\qquad{}+O(r^{d+2}l^{-\frac{d+2}{2}})+O(r^dl^{-\frac{d+2}{2}}),
\end{align*}
where $\s^2$ is the variance of $Z$.
The latter can be verified using Taylor's theorem for the density function of $Z$.

Then Theorem \ref{theorem:RW} applied to the functions $f_{r,z}$ and $f_{r(1+1/l)^{-1},z}$
gives the claim.
\end{proof}

%%%%%%%%%%%%%%%%%%%%%%%%%%%%%%%%%%%%%%%%%%%%
\section{Self-similar measures}
\label{section:ssproof}
%%%%%%%%%%%%%%%%%%%%%%%%%%%%%%%%%%%%%%%%%%%%

Let $\eta$ be a probability measure supported on contractive similarities of $\R^d$,
and let $\nu$ be the unique $\eta$-stationary measure.
Throughout this section, we assume that the set of contractions
on which $\eta$ is supported has no common fixed point.
Write $\mu=g(\eta)$, where we recall that $g$ is the ``projection'' 
\[
 \l\cdot\s(x)+v \mapsto \s(x)+v.
 \]
 Recall the definition of the operator
\[
Tf(\s)=\int f(\t(\k)^{-1}\s)\igap d\mu(\k).
\]

In this section, we apply our results to prove smoothness of $\nu$ if the contraction factors of the
similarities in the support of $\eta$ are sufficiently close to 1.
First we discuss the special case, when $\l(\mu)$ is supported on a single number.
In this case, we are able to give better quantitative bounds:

\begin{thm}\label{theorem:onelambda}
There is a number $c>0$ depending only on $d$ such that the following holds.
Let $v_1,v_2\in\R^d$ be points for which
\[
N^2:=\int |\k(v_1)-v_2|^2\igap d\mu(\k)
\]
is minimal.
Suppose that
\[
\int |\k(v_1)-v_2|^3\igap d\mu(\k)
\le MN^3.
\]
for some number $M\le\infty$.
Suppose further that there is a number $\l$ such that $\l=\l(\k)$ for $\eta$-almost every $\k$.
Then $\nu$ is absolutely continuous with $n$ times differentiable density if
\[
\l\ge1-c\frac{1-\|T\|_{L^2_0(\SO(d))}}{nM^2}.
\]
\end{thm}

\begin{proof}
Since $\nu$ is $\eta$-stationary, we have $\eta.\nu=\nu$ with a notation analogous to
\eqref{equation:mudotnu}.
If we take the Fourier transform of both sides in the above identity, then we can
derive the formula
\be\label{equation:onelambda}
\Res_r(\wh\nu)=S_r\Res_{\l r}(\wh\nu)
\ee
similarly to \eqref{equation:Sr}.
We can use this to express the Fourier transform of $\nu$ on the sphere of radius $r$
in terms of itself on the sphere of radius $\l r$.
This is the basis of our argument.

We note that $\nu$ is absolutely continuous with $n$ times differentiable density
if, say, 
\[
\int |\wh\nu(\xi)|(1+|\xi|)^n\igap d\xi<\infty.
\]
By simple computation, this will follow at once if we show that, say,
\be\label{equation:Fourierbound}
\|\Res_r(\wh\nu)\|_2\le Cr^{-(d+n+1)}.
\ee

Note that
\[
\int |g(\k)(\l v_1)-v_2|^3\igap d\eta(\k)=\int |\k( v_1)-v_2|^3\igap d\eta(\k)
\]
and a similar relation holds for the third moments.
Therefore, by Corollary \ref{corollary:general}, we have
\[
\|S_r\|\le1-c\frac{1-\|T\|_{L^2_0(\SO(d))}}{M^2}
\]
for sufficiently large $r$ with $c$ depending only on $d$.
We apply this and \eqref{equation:onelambda}
$\log r/\log\l^{-1}$ times for radii in the geometric progression
$r,\l r,\l^2 r,\ldots$ and conclude
\eqref{equation:Fourierbound} and hence the theorem.
\end{proof}

Now we turn to the more general case considered in Theorem \ref{theorem:selfsimilar},
when $\eta=p_1\d_{\k_1}+\ldots+p_k\d_{\k_k}$ and $\k_i$ may have different contraction ratios. Let $p_{min}=\min_i p_i$.
Then similarly to \eqref{equation:onelambda} we can write
\be\label{equation:morelambda}
\Res_r(\wh\nu)=\sum_{i=1}^k p_i\rho_r(g(\k_i))\Res_{\l_i r}(\wh\nu)
\ee
for the Fourier transform of the self-similar measure $\nu$.

We would like to apply Corollay \ref{corollary:general} to prove a norm
estimate for one of the operators $\rho_r(\k_i)$.
This is bound to fail, unfortunately; we need to take the average of several $\rho_r(\k_i)$
to have such an estimate.
Note however, that $\nu$ is also $\eta^{*(l_0)}$-stationary for all integers $l_0$.
We will consider an analogue of \eqref{equation:morelambda} for a decomposition of $\eta^{*(l_0)}$
with respect to contraction factors.
Since $(\R^+,\cdot)$ is commutative, but the group of similarities is not, we obtain many different
similarities in the support of $\eta^{*(l_0)}$ with the same contraction factors. This is exploited in the following proposition, which extracts from a sufficiently high convolution power $\eta ^ {*(l _ 0)}$
a piece which has the same contraction ratio and so that the corresponding measure on $\Isom (\R ^ d)$ has a spectral gap.

\begin{prp}\label{proposition:small piece}
There are $l _ 1 \in \N ^ +$, $ c > 0$ depending on $\norm {T} _ {L ^2 _ 0 (\SO (d))}$, $d$ and~$k $, and $q_0>0$ depending on these parameters and $p_{min}$, so that $\eta ^ {*(l _ 1)}$ can be written as $q _ 0 \eta _ 0+ (1 - q _ 0) \eta _ 1$ with $\eta _ 0, \eta _ 1$ probability measures on the semigroup of contracting similarities of $\R ^ n$, with all contractions appearing in the support of $\eta _ 0$ having the same contraction ratio, and if $\mu _ 0 = g (\eta _ 0)$ the corresponding measure on $\Isom (\R ^ d)$ the operator $\rho _ r (\mu _ 0)$ satisfies
\begin{equation*}
\norm {\rho _ r (\mu _ 0)} _ {L (S ^ {d -1})} \le 1 - c \min (1, r ^2)
.\end{equation*}
\end{prp}

The main ingredient in the proof of Proposition~\ref{proposition:small piece} is a following useful result of Mikl\'os Ab\'ert:

\begin{AbcTheorem}[\cite{Abe-spectral-gap}*{Corollary 3}]\label{theorem:Abert}
Let $\a$ be a probability measure on $\SO(d)$.
Let $\a=q_0\a_0+(1-q_0)\a_1$, where $\a_0$ and $\a_1$ are probability measures and $0\le q_0\le1$.
Suppose that $\|\cR_0(\a)\|\le q_0/2$.
Then
\[
\|\cR_0(\a_0)\|\le1-c\left(\frac{q_0}{\log q_0^{-1}}\right)^2,
\]
where $c$ is an absolute constant.
\end{AbcTheorem}

 Recall that $\cR$ denotes the regular representation of $\SO(d)$, and $\cR_0$ is  its restriction to the
subspace orthogonal to the constants.
Note that~$T=\cR(\t(\eta))$.

Using Theorem~\ref{theorem:Abert} and
comparing the exponential decay of the sequence
$\|\cR_0(\t(\eta ^ {*(l)}))\|$ with the
polynomial growth in the multiplicative group $(\R^+, \cdot) $ it is quite straightforward to find a decomposition $\eta ^ {*(l _ 0)} = q _ 1 \eta _ 1+ (1 - q _ 1) \eta _ 2$ as above so that $\theta (\eta _ 1)$ (the projection of $\eta _ 1$ to $\SO (d)$) has a spectral gap, i.e.\ such that
$\norm{\cR_0(\theta(\eta_1))}<1$
and the similarities in the support of $\eta_1$ have the same contraction factors.

We describe this decomposition in detail.
Write
\[
I:=\{\ua=(a_1,\ldots,a_k)\in\Z^k:a_1+\ldots+a_k=l_0, a_i\ge0\}.
\]
In addition, we write $\ula^\ua=\l_1^{a_1}\cdots\l_k^{a_k}$,
\[
\up^\ua=p_1^{a_1}\cdots p_k^{a_k}\frac{l_0!}{a_1!\cdots a_k!}
\]
for $\ua\in I$.
Let
\[
J_\ua:=\{\ub=(b_1,\ldots,b_{l_0})\in\Z^{l_0}:\#\{j:b_j=i\}=a_i,\text{for all $1\le i\le k$}\}
\]
and write
\[
\ueta^{*(\ua)}=\frac{a_1!\cdots a_k!}{l_0!}
\sum_{\ub\in J_\ua}\d_{\k_{b_1}\cdots\k_{b_{l_0}}}.
\]

With this notation, we have
\[
\eta^{*(l_0)}=\sum_{\ua\in I} \up^\ua\cdot\ueta^{*(\ua)}
\]
and $\l(\k)=\ula^{\ua}$ for every $\k\in\supp \ueta^{*(\ua)}$.
As $\absolute I \leq (l _ 0)^k$ and since
\begin{equation*}
\norm {T} = \norm {\mathcal{R} _ 0 (\theta(\mu))}<1
\end{equation*} 
(where as before $\mu = g(\eta)$) once $l_0$ is large enough
\begin{equation*}
\absolute {I} ^{-1} \geq 2 \mathcal{R} _ 0 (\theta(\mu ^ {*(l _ 0)})).
\end{equation*}
For such $l _ 0$ since $\sum_ {\ua \in I} \up^\ua=1$ there is a $\ua \in I$ such that \[\up^\ua \geq 2 \mathcal{R} _ 0 (\theta(\mu ^ {*(l _ 0)}))\]
 hence writing $\eta _ 0 = \ueta ^{\ua}$ and defining $\eta _ 1$ by
\begin{equation*}
\eta ^ {*(l_0)}=\up^\ua \eta _ 0 + (1 - \up^\ua) \eta _ 1
\end{equation*}
we may apply Theorem~\ref{theorem:Abert} to conclude that $\mu_0=g(\eta_0)$ satisfies
\begin{equation}\label{equation:spectralgap}
\norm {R_0(\theta(\mu _ 0))}<1-c(l_0)^{-3k}
.\end{equation}
Note that for all $\kappa\in\supp\eta_0$, the contraction ratio $\lambda(\kappa)=\ula^\ua$.

At this point we would like to apply Corollary~\ref{corollary:general} and conclude that $\norm {\rho _ r (\mu _ 0)} _ {L (S ^ {d -1})} < 1 - c \min (1, r ^2)$; however, to do so, we need first to establish that the isometries in the support of $\mu _ 0$ do not have a common fixed point (preferably in a quantifiable form).

The two measures on $\Isom (\R ^ d)$,\ $g (\eta ^ {*(l_0)})$ and $\mu ^ {*(l_0)}$ are in general distinct (though their projection to $\SO (d)$ coincides) since $g$ is \emph{not} a homomorphism.
The latter measure $\mu ^ {*(l _ 0)}$ has been studied extensively above and one way to conclude that the isometries of $\supp \mu _ 0$ do not have a common fixed point is by exploiting the relation between these two measures. We give an alternative proof below in detail, but first give a sketch of this argument. 

The support of the measure  $\eta _ 0$ is given by $\kappa _ 1  \kappa _ 2  \dots  \kappa _ {l _ 0}$ for a set of $(\kappa _ 1, \kappa _ 2, \dots, \kappa _ {l_0})$ of $\eta \times \dots \times \eta$ measure $\up^\ua$. 
We can define a new measure $\mu ' _ 0$ on $\Isom (\R ^ d)$ distinct from $\mu _ 0$ by taking each such $l_0$-tuple and sending it to $g (\kappa _ 1)  \dots  g (\kappa _ {l_0})$.
It follows from Proposition~\ref{proposition:decay}, which  has been a key ingredient in our analysis of the spectral radius of $\rho _ r (\mu)$, that for any $x, y \in \R ^ d$ the set of isometries $\kappa$ mapping $x$ to $y$ has $\mu ^ {*(l)}$ measure which is exponentially small in $l$, hence if $l _ 0$ was large enough there will be no common fixed point to all isometries in the support of $\mu ' _ 0$, and neither can there be a point which is nearly fixed by all these isometries. 
Therefore, if the contraction ratios $\lambda (\kappa)$ for $\kappa \in \supp \eta$ are all sufficiently close to one (in a way that ultimately depends only on the spectral gap for $T$, the cardinality $k$ of $\supp \eta$ and $p_{min}$) the isometries of $\mu_1$ also have no common fixed point, and hence Corollary~\ref{corollary:general} applies establishing Proposition~\ref{proposition:small piece}.

By working with a larger convolution power we can employ the following alternative argument to give an explicit proof of Proposition~\ref{proposition:small piece} that (though we do not work out the details here) gives better bounds.

Instead of applying Corollary \ref{corollary:general} directly to $\ueta^{*(\ua)}$,
we will show below in a series of Lemmata that we can either apply the corollary 
to $\underline\eta^{*(\underline a)}*\d_{\k_i}*\underline\eta^{*(\underline a)}$ 
or to $\underline\eta^{*(\underline a)}*\d_{\k_1\k_i}*\underline\eta^{*(\underline a)}$ with a suitable
choice of $1\le i\le k$. We assume below that the $\lambda_i$ are sufficiently close to one so that $\ula^\ua >1/2$.

First we record some simple but useful identities.

\begin{lem}\label{lemma:randomvector}
Let $U\in\R^d$ be a random vector.
Then
\[
\E[|U-v|^2]=\E[|U-\E[U]|^2]+|v-\E[U]|^2
\]
for all $v\in\R^d$ and
\[
\E[|U-\E[U]|^2]
=\E[|U|^2]-|\E[U]|^2.
\]
\end{lem}
\begin{proof}
For the first claim, we write
\begin{align*}
&\E[|U-v|^2]=\E[\langle(U-\E[U])-(\E[U]-v),(U-\E[U])-(\E[U]-v)\rangle]\\
&\qquad{}=\E[|U-\E[U]|^2]+|\E[U]-v|^2-2\Re(\E[\langle U-\E[U],\E[U]-v\rangle]).
\end{align*}
Since the third term vanishes, this proves the claim.

For the second part we write
\begin{align*}
\E[|U-\E[U]|^2]&=\E[|U|^2]+|\E[U]|^2-2\Re(\E[\langle U,\E[U]\rangle])\\
&=\E[|U|^2]-|\E[U]|^2.
\end{align*}
\end{proof}

\begin{lem}\label{lemma:etaa}
Let $v_1,v_2\in\R^d$ be two vectors such that
\[
\int |\kappa(v_1)-v_2|^2 \igap d\ueta^{*(\ua)}(\kappa)
\]
is minimal.
Then
\[
\int |\kappa(u_1)-u_2|^2 \igap d\ueta^{*(\ua)}(\kappa)\ge cl_0^{-3k}(|v_1-u_1|^2+|v_2-u_2|^2)
\]
for all $u_1,u_2\in\R^d$
\end{lem}
We note that the proof (only) uses the spectral gap property \eqref{equation:spectralgap}
about $\theta(\ueta^{*(\ua)})$.
\begin{proof}
Write
\[
E(u_1)=\int \kappa(u_1) \igap d\ueta^{*(\ua)}(\kappa)
\]
By Lemma \ref{lemma:randomvector}, we have
\[
\int |\kappa(u_1)-u_2|^2 \igap d\ueta^{*(\ua)}(\kappa)\ge\int |\kappa(u_1)-E(u_1)|^2 \igap d\ueta^{*(\ua)}(\kappa)=:F(u_1).
\]
Clearly, $F(u_1)$ is a polynomial of degree at most two in $u_1$, and it takes its minimum at $v_1$.
Thus
\[
F(u_1)=F(v_1)+\lim_{t\to\infty}\frac{F(t (u_1-v_1))}{t^2}.
\]

Recall that $\l(\k)=\ula^{\ua}$ for $\ueta^{*(\ua)}$-almost every $\kappa$, and observe that the translation
parts of $\k$ is negligible, when we evaluate it on a long vector.
Then
\begin{align*}
&\lim_{t\to\infty}\frac{F(t(u_1-v_1))}{t^2}\\
&\quad{}=\ula^{2\ua}\int \left|\s_1(u_1-v_1)-\int \s_2(u_1-v_1) \igap d\t(\ueta^{*(\ua)})(\s_1)\right|^2 \igap d\t(\ueta^{*(\ua)})(\s_2)\\
&\quad{}=\ula^{2\ua}\left(|u_1-v_1|^2-\left|\int \s_2(u_1-v_1) \igap d\t(\ueta^{*(\ua)})(\s_1)\right|^2\right)
\end{align*}
by the second part of  Lemma \ref{lemma:randomvector}.

Using the spectral gap property \eqref{equation:spectralgap} we get
\begin{align*}
\left|\int \s_2(u_1-v_1) \igap d\t(\ueta^{*(\ua)})(\s_1)\right|&\le\|\cR_0(\t(\ueta^{*(\ua)}))\|\cdot|u_1-v_1|\\
&\le(1-cl_0^{-3k})|u_1-v_1|.
\end{align*}
Thus
\[
\lim_{t\to\infty}\frac{F(t(u_1-v_1))}{t^2}\ge cl_0^{-3k} |u_1-v_1|^2.
\]
Here we used $\ula^{\ua}\ge1/2$.

We proved that
\[
\int |\kappa(u_1)-u_2|^2 \igap d\ueta^{*(\ua)}(\kappa)\ge cl_0^{-3k} |u_1-v_1|^2.
\]
If we apply the same argument to
\[
\int |\kappa(u_1)-u_2|^2 \igap d\ueta^{*(\ua)}(\kappa)
=\ula^{2\ua}\int |\kappa^{-1}(u_2)-u_1|^2 \igap d\ueta^{*(\ua)}(\kappa),
\]
we get
\[
\int |\kappa(u_1)-u_2|^2 \igap d\ueta^{*(\ua)}(\kappa)\ge cl_0^{-3k} |u_2-v_2|^2.
\]
This together with the previous bound proves the lemma.
\end{proof}

\begin{lem}\label{lemma:composition}
Let $v_1,v_2$ be the same as in Lemma \ref{lemma:etaa}.
Then for every $u_1,u_2\in\R^d$ and similarity $\k_0$, we have
\[
\int|\kappa(u_1)-u_2|^2
\igap d\ueta^{*(\ua)}*\d_{k_0}*\ueta^{*(\ua)}(\kappa)
\ge cl_0^{-6k}|\kappa_0(v_2)-v_1|^2.
\]
\end{lem}
This lemma provides us with a bound on the second moment needed to apply Corollary
\ref{corollary:general} for the measure $\ueta^{*(\ua)}*\d_{k_0}*\ueta^{*(\ua)}(\kappa)$
provided $\k_0$ does not map $v_2$ near $v_1$.
In the proof of Proposition \ref{proposition:small piece}, we will find such an element $\k_0$
among $\k_1,\ldots,\k_k,\k_1\k_1,\ldots,\k_k\k_1$.
\begin{proof}
Using Lemma \ref{lemma:etaa} twice we write
\begin{align*}
&\int|\kappa(u_1)-u_2|^2
\igap d\ueta^{*(\ua)}*\d_{\k_0}*\ueta^{*(\ua)}(\kappa)\\
&\qquad{}=\ula^{2\ua}\l(\k_0)^2\int|\kappa_1(u_1)-\k_0^{-1}\kappa_2^{-1}(u_2)|^2
\igap d\ueta^{*(\ua)}(\kappa_1) d\ueta^{*(\ua)}(\kappa_2)\\
&\qquad{}\ge cl_0^{-3k}\int|v_2-\kappa_0^{-1}\kappa^{-1}(u_2)|^2
\igap d\ueta^{*(\ua)}(\kappa)\\
&\qquad{}= cl_0^{-3k}\l(\k_0)^{-2}(\ula^{2\ua})^{-1}\int|\kappa(\kappa_0(v_2))-u_2)|^2
\igap d\ueta^{*(\ua)}(\kappa)\\
&\qquad{}\ge cl_0^{-6k}|\kappa_0(v_2)-v_1|^2,
\end{align*}
which was to be proved.
\end{proof}

\begin{proof}[Proof of Proposition~\ref{proposition:small piece}]
We assume without loss of generality that \linebreak
$\max_{i=1,\ldots,k}|v(\k_i)|=1$, and for every
$x\in\R^d$, there is $1\le i\le k$ such that $|\k_i(x)-x|\ge 1$.
Indeed, this will hold if we change the origin and rescale the metric.
Now let $v_1, v_2$ be as in Lemma \ref{lemma:etaa}.
If there is $1\le i\le k$ such that $|\k_i(v_1)-v_2|\ge1/2$, then let $\k_0=\k_i$, $q_0=\up^{2\ua}p_i$
and $l_1=2l_0+1$.
In the opposite case, let $1\le i\le k$ be such that $|\k_i(\k_1(x))-\k_1(x)|>1$, and we take
$\k_0=\k_i\k_1$, $q_0=\up^{2\ua}p_ip_1$ and $l_1=2l_0+2$.
Observe that $|\k_0(v_1)-v_2|\ge1/2$ in both cases.

Take $\eta_0=\ueta^{*(\ua)}*\d_{\k_0}*\ueta^{*(\ua)}$.
We apply Corollary \ref{corollary:general} to the measure $\mu_0=g(\eta_0)$.
Clearly,
\[
\|\cR_0(\t(\mu_0))\|\le\|\cR_0(\t(\ueta^{*(\ua)}))\|\le1-cl_0^{3k}.
\]

Let $u_1,u_2\in\R^d$ be such that
\[
\int|\k(u_1)-u_2|^2\igap d\eta_0(\k)
\]
is minimal.
Observe that the above quantity is at most $(2l_0+2)^2$ for $u_1=u_2=0$ by our choice of the
coordinate system.
By the same argument as in the proof of Lemma \ref{lemma:etaa}, we can show that
\[
\int|\k(0)-0|^2\igap d\eta_0(\k)\ge cl_0^{-3k}(|u_1-0|^2+|u_2-0|^2).
\]
We conclude $|u_1|^2+|u_2|^2\le Cl_0^{3k+2}$ from these.
Thus
\[
\int|\k(u_1)-u_2|^3\igap d\eta_0(\k)\le C (l_0^{(3k+2)/2}+l_0+l_0^{(3k+2)/2})^3\le C l_0^{6k}.
\]

By Lemma \ref{lemma:composition}, we have
\[
\int|\k(u_1)-u_2|^2\igap d\eta_0(\k)\ge c l_0^{-6k}.
\]

Since $\l(\k)=\ula^{2\ua}\cdot\l(\k_0)=:\l_0$ for $\eta_0$-almost all $\k$, the minimum of
\[
\int|g(w_1)-w_2|^2\igap d\mu_0(g)
\]
is attained for $w_1=\l_0 u_1$ and $w_2=u_2$.
Moreover,
\begin{align*}
\int|g(w_1)-w_2|^2\igap d\mu_0(g)\ge c l_0^{-6k}&=:N^2\quad{\rm and}\quad\\
\int|g(w_1)-w_2|^3\igap d\mu_0(g)\le Cl_0^{6k}&=C l_0^{15k}N^2.
\end{align*}
Therefore
\[
\|\rho_r(\mu_0)\|\le 1-c l_0^{-33k}\min(1,r^2)
\]
by Corollary \ref{corollary:general}.
\end{proof}

Given Proposition~\ref{proposition:small piece}, the proof of Theorem~\ref{theorem:selfsimilar} for the general case proceeds as in Theorem~\ref{theorem:onelambda}:

\begin{proof}[Proof of Theorem~\ref{theorem:selfsimilar}]
By Proposition~\ref{proposition:small piece} we can write $\eta^{*(l_1)}=\sum_{i=0}^A q_i\eta_i$ for some integer $A$,
measures $\eta_i$ and positive real numbers $q_i$
such that $\l(\eta_i)=\l_i$ is constant almost surely and $\|\rho_r(g(\eta_0))\|\le 1-c \min(1,r^2)$.

We recall \eqref{equation:morelambda}:
\[
\Res_r(\wh\nu)=\sum_{i=0}^A q_i\rho_r(g(\eta_i))\Res_{\l_i r}(\wh\nu).
\]
Since $\l(\k_i)\ge \bar\l$ for all $1\le i\le k$, we have
\[
\|\Res_r(\wh\nu)\|_2\le(1-q_0+q_0\|\rho_r(\mu_0)\|)\max_{r>s>\bar\l^{l_1}r}\{\|\Res_s(\wh\nu)\|_2\}.
\]
By induction, this implies that $\|\Res_r(\wh\nu)\|_2$ has arbitrarily fast polynomial decay
if $\bar\l^{l_1}$ is sufficiently close to $1$.
Note that $l_1$ and $q_0$ depends only on $k$, $p_{min}$ and the spectral gap of $T$.
This proves the theorem.
\end{proof}

\bibliography{varju}
\bibliographystyle{alpha}

\end{document}